\documentclass[10pt,english,reqno]{amsart}
\usepackage{amsmath,amssymb,amsthm,amsfonts,graphicx,color}
\usepackage{amssymb,geometry}

\usepackage{color, graphics}

\usepackage{graphicx}
\usepackage[T1]{fontenc}
\usepackage[latin1]{inputenc}
\usepackage[english]{babel}
\usepackage{geometry}
\usepackage{cases}

\usepackage[colorlinks=true]{hyperref}

\usepackage[colorlinks=true]{hyperref}
\hypersetup{linkcolor=black,citecolor=black,filecolor=black,urlcolor=black} 

\newcommand{\red}[1]{{\color{red} #1}}

\newcommand{\h}{{\tt h}}
\newcommand{\q}{{\tt q}}
\newcommand{\pe}{{\tt p}}
\newcommand{\bi}{{\tt b}}
\newcommand{\ta}{{\tt a}}
\newcommand{\x}{{\tt x}}
\newcommand{\y}{{\tt y}}
\newcommand{\s}{{\tt s}}
\newcommand{\z}{{\tt z}}

\newcommand{\w}{{\tt w}}

\newcommand{\Q}{\mathbb{Q}}
\newcommand{\R}{\mathbb{R}}

\newcommand{\diver}{\text{div}}

\newcommand{\eps}{\varepsilon}

 \newtheorem{lemma}{Lemma}[section]

\newtheorem{theorem}{Theorem}[section]
\newtheorem{proposition}{Proposition}[section]

\newtheorem{remark}{Remark}[section]
\newcommand{\bremark}{\begin{remark} \em}
\newcommand{\eremark}{\end{remark} }

\numberwithin{equation}{section}

\title[\tiny{k-ended $O(m)\times O(n)$ invariant solutions to the Allen-Cahn equation}]{k-ended $O(m)\times O(n)$ invariant solutions to the Allen-Cahn equation with infinite Morse index}

\author{Oscar Agudelo}
\address{University of West Bohemia in Pilsen-NTIS, Univerzitn\'{i} 22, Czech Republic.}
\email {oiagudel@ntis.zcu.cz}

\author{Matteo Rizzi}
\address{Mathematisches Institut, Justus Liebig Universit\"{a}t, Arndtstrasse 2, 35392, Giessen, Germany.}
\email{mrizzi1988@gmail.com}
\thanks{O. Agudelo was supported by the Grant 18-032523S
of the Grant Agency of the Czech Republic and also by the Project LO1506 of the Ministry
of Education, Youth and Sports of the Czech Republic. M. Rizzi was partially supported by the Alexander von Humboldt foundation.}

\begin{document}

\setlength{\parskip}{1pt}

\maketitle

\begin{abstract}
In this work we study existence, asymptotic behavior and stability properties of $O(m)\times O(n)-$invariant solutions of the Allen-Cahn equation $\Delta u +u(1-u^2)=0$ in $\R^m\times\R^n$ with $m,n\geq 2$ and $m+n\geq 8$. We exhibit four families of solutions whose nodal sets are smooth logarithmic corrections of the Lawson cone and with infinite Morse index. This work complements the study started in \cite{PW} by Pacard and Wei and \cite{AKR1} by Agudelo, Kowalckzyk and Rizzi.
\end{abstract}

\tableofcontents

\section{Introduction}
We consider the Allen-Cahn equation
\begin{equation}
\label{eq_AC}
-\Delta u=u-u^3 \quad \hbox{in} \quad \R^{N+1}. 
\end{equation}

Equation \eqref{eq_AC} arises in the context of the gradient theory of phase transitions \cite{AC} describing the behaviour of two liquids which are mixed in a container, or two different materials in a binary alloy.

In the seventy's, the developments in the $\Gamma$-convergence theory, (see \cite{Mo,MoMo}) showed the strong and still fruitful relation between equation \eqref{eq_AC} and the theory of {\it minimal surfaces.} A famous conjecture raised by E. De Giorgi in 1978 \cite{DeG} asserts that, at least when $N\le 7$, if $u:\R^{N+1}\to\R$ is a bounded solution of \eqref{eq_AC} which is monotone in one direction, then $u$ is one dimensional, i.e., $u$ depends only on one euclidean variable. In other words, $$
u(\xi)=v_\star(\nu\cdotp \xi+c) \quad \hbox{for } \xi\in \R^{N+1},$$
where $\nu\in S^N$ and $c\in\R$ are fixed and $v_\star(t):=\tanh(t/\sqrt{2})$ is the unique solution to the corresponding ODE problem
\begin{equation}
\label{EDO-AC}
-v''_\star=v_\star-v_\star^3,\qquad v_\star(0)=0,\qquad v_\star(\pm\infty)=\pm 1.
\end{equation}
The De Giorgi conjecture is also motivated by the Bernstein conjecture, which states that any entire {\it minimal graph} over $\R^N$ is affine if $N\le 7$ ( see \cite{Be}). The Bernstein conjecture is known to be true and sharp about the dimension, in fact in \cite{BDG} the authors construct an entire minimal graph $\Gamma$ over $\R^{2n}$ which is not affine, for any $n\ge 4$.

The De Giorgi conjecture is known to be true in dimension $N+1=2$ (see \cite{GG}) and $N+1=3$ (see \cite{AmCa,FSV}). Under the additional assumption that
$$\lim_{\xi_1\to\pm\infty} u(\xi_1,\xi')=\pm 1\qquad \forall\,\xi'\in\R^N,$$
one-dimensional symmetry was proved in dimension $4\le N+1\le 8$ in \cite{Sa}. In this case too, the conjecture is sharp about the dimension, since in dimension $N+1=9$ there exists a family of solutions to (\ref{eq_AC}) which are monotone in one direction but not one-dimensional (see \cite{DKW}). In fact, the zero level set is close to a dilated version of the minimal graph $\Gamma$ constructed in \cite{BDG}.

The monotonicity assumption in De Giorgi's conjecture is related to the stability of the solution and it is strongly connected to the stability properties of the nodal set of the solutions. Stability is a natural tool to classify bounded solutions of equation \eqref{eq_AC}. In dimension $N+1=2$, stable entire solutions are one-dimensional (see \cite{FSV}).

Stable entire solutions to the Allen-Cahn equation in $\R^{N+1}$ are expected to be one-dimensional if $N\le 6$, this is related to the fact that the Lawson cones
\begin{equation}
\label{Lawson_cones}
C_{m,n}:=\left\{(x,y)\in\R^m\times\R^n:\,(n-1)|x|^2=(m-1)|y|^2\right\},\qquad m,\,n\ge 2
\end{equation}
are stable provided $N+1=m+n\ge 8$ (see \cite{AKR1,Da}), while there are no minimising cones embedded in $\R^{N+1}$ if $N\le 6$.

In the case $N\ge 7$ stable solutions which are not one-dimensional are known to exist (see \cite{PW}). The nodal sets of these  solutions are close to stable (actually minimising) minimal hypersurfaces, which are asymptotic to the Lawson cone $C_{m,n}$ in the case $m+n\ge 9$ or $m+n=8$ and $m,\,n\ge 3$.

As for unstable solutions, the authors in \cite{ADW} exhibit two families of axially symmetric bounded solutions in $\R^3$. One of them having Morse Index one and the other having large Morse Index.

In \cite{AKR1}, the authors consider the case $N+1=m+n\ge 8$ with $m,\,n\ge 3$ and generalise the existence result in \cite{PW}. More precisely, in \cite{AKR1} the authors show solutions satisfying that
\begin{equation}
u(\sigma x,\gamma y)=u(x,y),\qquad\forall\, (x,y)\in\R^m\times \R^n,\,(\sigma,\gamma)\in O(m)\times O(n)
\end{equation}
and whose zero level set is the disjoint union of two connected components that are smooth logarithmic corrections from the Lawson cone $C_{m,n}$.

Our first main result in this work is the following.

\begin{theorem}
\label{main-th-AC}
Let $m,n\ge 2$, $N+1=n+m\ge 8$. Then there exists $\eps_0>0$ such that, for any $\eps\in(0,\eps_0)$, there exist a bounded solution $u_\eps$ to the Allen-Cahn equation (\ref{eq_AC}) such that
\begin{itemize}

\item[i.] \label{geom-inv} $u_\eps$ is smooth and $O(m)\times O(n)$-invariant;
\item[ii.]\label{k-end} the zero level set of $u_{\eps}$ is the disjoint union of $k$ connected components;

\item[iii.]\label{energy-estimate} there exists a constant $C>0$ such that, for any $\eps\in(0,\eps_0)$ and $R>2\eps^{-1}$, 
\begin{equation}
\label{energy-est}
\int_{B_R}\left(\frac{1}{2}|\nabla u_\eps|^2+\frac{1}{4}(1-u_\eps^2)^2\right) \le C R^N
\end{equation}

\item[iv.]\label{Morse-index} the Morse index of $u_\eps$ is infinite.
\end{itemize}
\end{theorem}

\begin{remark}
Theorem \ref{main-th-AC} generalizes the  results in \cite{AKR1}, by exhibiting solutions with nodal sets having $k\geq 2$ components diverging from the Lawson cone with logarithmic growth (see Theorems \ref{th_Sigma} and \ref{main-th-JT} below) and where either $n=2$ or $m=2$ are allowed. 

Moreover, the Morse index of these solutions, as well as the solutions in \cite{AKR1}, is infinite.

Another consequence of Theorem \ref{main-th-AC} is that in higher dimensions, the energy estimate (\ref{energy-est}) does not imply that the Morse index is finite.
\end{remark}

\begin{remark}
Let $u_{eps}$ be solution predicted in Theorem \ref{main-th-AC}. Then, $-u_\eps$ also satisfies properties ${\rm i. - iv.}$ in Theorem \ref{main-th-AC}. On the other hand, switching the roles of $m$ and $n$, it is possible to construct (in the same manner as for $u_{\eps}$ and $-u_{\eps}$) two more solutions $v_\eps$ and $-v_{\eps}$ of $O(n)\times O(m)$ invariant solutions to (\ref{eq_AC}) satisfies properties ${\rm i. - iv.}$ in Theorem \ref{main-th-AC}.

Composing with the reflection $\sigma(x,y)=(y,x)$, defined for $(x,y)\in\R^m\times\R^n$, we get two more families of solutions $\pm w_\eps:=\pm v_\eps\circ\sigma$ to (\ref{eq_AC}) fulfilling properties (\ref{geom-inv}), (\ref{k-end}) and (\ref{energy-estimate}) of Theorem \ref{main-th-AC}.
\end{remark}

In order to prove Theorem \ref{main-th-AC}, we first find an approximation of the candidate of nodal set. This candidate is in turn, based on suitable $O(m)\times O(n)$-invariant hypersurfaces, which are asymptotic to the Lawson cone $C_{m,n}$ at infinity, with $m,n\ge 2$, $m+n\ge 8$. Our next result deals with existence and stability of such surfaces.

Denote
$$
E^+_{m,n}:=\left\{(x,y)\in\R^m\times\R^n:\,(n-1)|x|^2<(m-1)|y|^2\right\}$$
$$E^-_{m,n}:=\left\{(x,y)\in\R^m\times\R^n:\,(n-1)|x|^2>(m-1)|y|^2\right\}.
$$

\begin{theorem}\label{th_Sigma}

Let $m,n\ge 2$, $m+n\ge 8$. Then there exist two unique minimal hypersurfaces $\Sigma^\pm_{m,n}\subset E^\pm_{m,n}$ asymptotic to the cone $C_{m,n}$ at infinity and satisfying that
\begin{enumerate} 
\item[(i)] $\Sigma^\pm_{m,n}$ are smooth;

\medskip
\item[(ii)]  $\Sigma^\pm_{m,n}$ are $O(m)\times O(n)$-invariant;

\medskip
\item[(iii)] ${\rm dist}(\Sigma^{\pm}_{m,n},\{0\})=1$;

\medskip
\item[(iv)] $\Sigma^\pm_{m,m}$ are stable.
\end{enumerate}
\end{theorem}

Theorem \ref{th_Sigma} is a generalisation of Theorem $1.1$ in \cite{AKR1}, which also deals with the cases $n=2$ and $m=6$ or $n=6$ and $m=2$. If $m+n\ge 9$ and $m,n\ge 2$ or $m+n=8$ and $m,n\ge 3$, stability follows from Theorem $2.1$ of \cite{HS}, which also provides the construction of such surfaces under the hypothesis that the tangent cone is minimising. The main contribution of  Theorem \ref{th_Sigma} is the proof of existence and stability in the cases $m=2$ and $n=6$ or $m=6$ and $n=2$, which is a non trivial new result. 


The next result provides the sets that will serve as main approximation of the nodal set of the solutions predicted in Theorem \ref{main-th-AC}.

\begin{theorem}
\label{main-th-JT}
Let {$m,\,n\ge 2$}, $m+n\ge 8$  and let $\gamma>0$. Let $\Sigma$ denote one of the two minimal hypersurfaces constructed in Theorem \ref{th_Sigma}. Then there exists $\delta_*>0$ small such that if $0<\delta \leq \delta_*$, the system of equations
\begin{equation}
\label{JT-system}
\delta \big(\Delta_\Sigma \h_j+|A_\Sigma|^2 \h_j\big)=  \gamma (e^{-\sqrt{2}(\h_j-\h_{j-1})}-e^{-\sqrt{2}(\h_{j+1}-\h_j)}),\qquad 1\le j\le k
\end{equation}
has a smooth solution $\h:=(\h_1,\dots,\h_{k})$ which is $O(m)\times O(n)$-invariant and such that
\begin{equation}
\label{disj-graph}
-\infty\equiv \h_0<\h_1<\h_2<\cdots<\h_k<\h_{k+1}\equiv+\infty.
\end{equation}
\end{theorem}

Observe that property (\ref{disj-graph}) in Theorem \ref{main-th-JT} guarantees that the graphs of the $\h_j$ are mutually disjoint.

The idea of the proof of Theorem \ref{main-th-AC} is that  the connected components of the zero level set of the solutions to (\ref{eq_AC}) are close to the graphs of the $\h_j$'s. The proof relies on an infinite dimensional Lyapunov-Schmidt reduction.

The paper is organised as follows. In Section 2 we prove Theorem \ref{th_Sigma} and study the Jacobi fields of the predicted surfaces. In Section 3 we prove Theorem \ref{main-th-JT} and study invertibility of its associated linearezed operator. In Section 4 we prove Theorem \ref{main-th-AC} and leave the calculations of the Morse Index of the solutions to Section 5.

\section{The surfaces $\Sigma^\pm_{m,n}$: construction and Jacobi fields}

The aim of this Section is to prove Theorem \ref{th_Sigma}, that is to construct the surfaces $\Sigma^\pm_{m,n}$, and to study the asymptotic behaviour at infinity of their Jacobi fields. We will actually prove more than Theorem \ref{th_Sigma}, that is we will prove the following Theorem.

\begin{theorem}\label{Th_Sigma_detailed}
Let $m,n\ge 2$, $m+n\ge 8$. Then there exist two unique minimal hypersurfaces $\Sigma^\pm_{m,n}\subset E^\pm_{m,n}$ asymptotic to the cone $C_{m,n}$ at infinity and satisfying properties (i), (ii) and (iii) of Theorem \ref{th_Sigma}. Moreover
\begin{equation}
\label{stab+}
\langle \pe,\nu_{\Sigma^+_{m,n}}(\pe)\rangle >0\qquad\forall\,\pe\in\Sigma^+_{m,n}
\end{equation}
and 
\begin{equation}
\label{stab-}
\langle \pe,\nu_{\Sigma^-_{m,n}}(\pe)\rangle <0\qquad\forall\,\pe\in\Sigma^-_{m,n},
\end{equation}
being $\nu_{\Sigma^\pm_{m,n}}(\pe)$ the unit normal to $\Sigma^\pm_{m,n}$ at $\pe$ pointing towards $E^+_{m,n}$. 
\end{theorem}
In the sequel, we will set $\Sigma:=\Sigma^-_{m,n}$ and $E^+:=E^+_{m,n}$. We set
$$\gamma^\pm=-\frac{N-2}{2}\pm\sqrt{\left(\frac{N-2}{2}\right)^2-(N-1)}.$$
\begin{remark}
\label{rem-Sigma}
\begin{enumerate}
\item The results in \cite{AA} and \cite{S} imply that, outside a ball, $\Sigma^\pm_{m,n}$ are normal graphs over the cone $C_{m,n}$ of $O(m)\times O(n)$-invariant functions ${\tt w}^\pm:C_{m,n}\to\R$ such that ${\tt w}^+>0$ and ${\tt w}^-<0$. Moreover, Theorem $4.2$ of \cite{PW} implies that the asymptotic behaviour at infinity of ${\tt w}^\pm$ is given by
$${\tt w}^\pm(\xi)=r^{\gamma}(c_\pm+O(r^{-\alpha})),\qquad\gamma\in\{\gamma^\pm\}, r:=|\xi|,\,\xi\in C_{m,n},\, c_+>0,\,c_-<0$$
where $\alpha>0$ and this relation can be differentiated.
\item If a minimal surface $\Sigma$ has a positive Jacobi field $u>0$, then it is stable. In fact, setting $w:=\log u$, we have
$$\Delta_\Sigma w=\frac{\Delta_\Sigma u}{u}-|\nabla_\Sigma w|^2=-|A_\Sigma|^2-|\nabla_\Sigma w|^2.$$
Taking a test function $\varphi\in C^\infty_c(\Sigma)$, multiplying by $\varphi^2$ and integrating by parts, we have
\begin{equation}\notag
\begin{aligned}
\int_\Sigma|A_\Sigma|^2\varphi^2&=-\int_{\Sigma}|\nabla_\Sigma w|^2\varphi^2+2\int_\Sigma \varphi \nabla w\cdotp\nabla\varphi\\
&\le-\int_{\Sigma}|\nabla_\Sigma w|^2\varphi^2+2\int_\Sigma |\varphi||\varphi \nabla w|\,|\nabla\varphi|\\
&\le-\int_{\Sigma}|\nabla_\Sigma w|^2\varphi^2+\int_\Sigma \varphi^2|\varphi \nabla w|^2+\int_\Sigma |\nabla\varphi|^2=\int_\Sigma |\nabla\varphi|^2.
\end{aligned}
\end{equation}
\item Point (iv) of Theorem \ref{th_Sigma} follows from Theorem \ref{Th_Sigma_detailed}. In fact $\Sigma^\pm_{m,n}$ have a positive Jacobi field given by the function $\pe\in\Sigma^\pm_{m,n}\mapsto\pm\langle \pe,\nu_{\Sigma^\pm_{m,n}}(\pe)\rangle$, so in particular $\Sigma^\pm_{m,n}$ are stable.
\end{enumerate}
\end{remark}

\subsection{Proof of Theorem \ref{Th_Sigma_detailed}}
We look for $O(m)\times O(n)$-invariant smooth surfaces, that is surfaces generated by a smooth, regular curve $\Upsilon:\R \to \R^2$, $\Upsilon(s):=(a(s),b(s))$ in the half-plane
$$
Q:= \{(a,b)\in \R^2\,:\, a>0\} 
$$
In other words, we want to construct surfaces $\Sigma$ such that, for any $\pe\in\Sigma$, there exists a unique $(s,\x,\y)=(s(\pe),\x(\pe),\y(\pe))\in\R\times S^{m-1}\times S^{n-1}$ such that
\begin{equation}
\label{param_O(m)O(n)_invariant}
\pe:=\left(a(s){\tt x},b(s){\tt y}\right).
\end{equation}
If we parametrise $\Upsilon$ by arc length, that is we impose the condition $(a')^2+(b')^2=1$, then the zero mean curvature equation translates into the ODE
\begin{equation}
\label{H=0-inv}
-a''b'+a'b''+(m-1)\frac{b'}{a}+(n-1)\frac{a'}{b}=0.
\end{equation}
Now we introduce the change of variables
\begin{equation}
\label{variables-uv}
\tan u=\frac{a}{b}\qquad\tan v=\frac{a'}{b'}
\end{equation}
also used in \cite{ABPRS}, so that equation (\ref{H=0-inv}) is equivalent to the dynamical system
\begin{equation}
\label{dyn-syst-uv}
\begin{aligned}
u'&=X_1(u,v)\\
v'&=X_2(u,v)
\end{aligned}
\end{equation}
where the vector field $X(u,v)=(X_1(u,v),X_2(u,v))$ is given by
\begin{eqnarray}
\begin{aligned}
X_1(u,v)&:=\cos u\sin u \sin(u-v)\\
X_2(u,v)&:=(m-1)\sin u\sin v-(n-1)\cos u\cos v
\end{aligned}
\end{eqnarray}
The aim now is to study the phase plane associated to vector field $X$, in order to determine all the associated orbits. By periodicity, we can reduce ourselves to consider this vector field in the rectangle $D:=(0,\frac{\pi}{2})\times(-\pi,\pi)$. Since $\bar{D}$ is compact and $X$ is continuous, it follows that every orbit is complete, that is defined in all $\R$. We note that
\begin{equation}
\label{sign-X_1}
\begin{aligned}
\{(u,v)\in\bar{D}:\,X_1(u,v)=0\}&=\{(u,v)\in\bar{D}:\,v=u-k\pi,\,k\in\{0,1\}\}\cup(\{0\}\times[-\pi,\pi])\cup
(\{\frac{\pi}{2}\}\times[-\pi,\pi])\\
\{(u,v)\in\bar{D}:\,X_1(u,v)>0\}&=\{(u,v)\in\bar{D}:\,u-\pi<v<u\}\\
\{(u,v)\in\bar{D}:\,X_1(u,v)<0\}&=\{(u,v)\in\bar{D}:\,v<u-\pi\}\cup\{(u,v)\in\bar{D}:\,v>u\}\\
\end{aligned}
\end{equation}
Similarly, setting
$$\bar{v}(u):=\arctan\left(\frac{n-1}{m-1}\cot u\right),$$
we have
\begin{equation}
\label{sign-X_2}
\begin{aligned}
\{(u,v)\in\bar{D}:\,X_2(u,v)=0\}&=\{(u,v)\in\bar{D}:\,v=\bar{v}(u)-k\pi,\,k\in\{0,1\}\}\\
\{(u,v)\in\bar{D}:\,X_2(u,v)>0\}&=\{(u,v)\in\bar{D}:\,v<\bar{v}(u)-\pi\}\cup\{(u,v)\in\bar{D}:\,v>\bar{v}(u)\}\\
\{(u,v)\in\bar{D}:\,X_2(u,v)<0\}&=\{(u,v)\in\bar{D}:\,\bar{v}(u)-\pi<v<\bar{v}(u)\}\\
\end{aligned}
\end{equation}
As a consequence, the singular points of $X$ in $\bar{D}$ are given by
\begin{equation}
p_1:=(0,-\frac{\pi}{2}),\, p_2:=(0,\frac{\pi}{2}),\, p_3:=(\frac{\pi}{2},-\pi),\, p_4=(\frac{\pi}{2},0),\,p_5=(\frac{\pi}{2},\pi),\,p_6:=(\bar{\alpha},\bar{\alpha}),\,p_7:=(\bar{\alpha},\bar{\alpha}-\pi),
\end{equation}
where $\bar{\alpha}:=\arctan\left(\sqrt\frac{n-1}{m-1}\right)$ is such that $p_6:=(\bar{\alpha},\bar{\alpha})$ is the intersection between the straight line $v=u$ and the graph of $\bar{v}$.
\begin{lemma}
For any integers $m,n\ge 2$, $m+n\ge 8$, the points $p_j$ are saddle points for $1\le j\le 5$, $p_6$ is an unstable node and $p_7$ is a stable node.
\end{lemma}
For the proof, see Proposition $3.4$ of \cite{ABPRS}. This is based on an explicit computation of the Jacobian matrix of $X$ at $p_j$.
\begin{lemma}
\label{lemma-per-orb}
Let $m,n\ge 2$, $m+n\ge 8$. Then the dynamical system (\ref{dyn-syst-uv}) admits no periodic orbits in $\bar{D}$.
\end{lemma}
\begin{proof}
If we assume by contradiction that $X$ admits a periodic orbit $\phi=(u,v)$ in $\bar{D}$, then, denoting the interior of $\phi$ by $\Omega$, we would have either $\Omega\subset (0,\frac{\pi}{2})\times (0,\frac{\pi}{2})$ or $\Omega\subset (0,\frac{\pi}{2})\times (-\pi,-\frac{\pi}{2})$. This follows from the phase plane study. Moreover, since
$$\diver X=(3 \cos^2 u+m-2)\sin u\cos v+(3 \sin^2 u+n-2)\sin v\cos u,$$
then we have $\diver X>0$ in $(0,\frac{\pi}{2})\times (0,\frac{\pi}{2})$ and $\diver X<0$ in $(0,\frac{\pi}{2})\times (-\pi,-\frac{\pi}{2})$. If, for instance, $\Omega\subset (0,\frac{\pi}{2})\times (0,\frac{\pi}{2})$, by the divergence Theorem we have
$$\int_{\partial\Omega} X\cdotp\nu d\sigma=\int_\Omega \diver X>0.$$
On the other hand, by the dynamical system (\ref{dyn-syst-uv}) we have $X=\phi'$ on $\partial\Omega$ and $\nu=(\phi')^\bot$, so that $X\cdotp\nu=0$ on $\partial\Omega$, a contradiction.
\end{proof}
 
Once proved Lemma \ref{lemma-per-orb}, the argument is the same as in \cite{ABPRS}. We will sketch it here for the sake of completeness. Using the phase plane of the vector field $X$ and computing the eigenvalues of $JX(p_j)$, we can see that the orbits intersecting $\bar{D}$ are either constants or heteroclinic. The heteroclinic ones can be classified as follows.
\begin{enumerate}
\item The ones which are contained in a vertical line, that is
\begin{equation}
\begin{aligned}
\{\phi_1(t)\}_{t\in\R}&=W^u(p_2)\cap W^s(p_1),\qquad\{\phi_2(t)\}_{t\in\R}=W^u(p_4)\cap W^s(p_3)\\
\{\phi_3(t)\}_{t\in\R}&=W^u(p_4)\cap W^s(p_5),\qquad\{\phi_4(t)\}_{t\in\R}=W^u(p_2)\cap W^s(p_2+(0,\pi))\\
\{\phi_5(t)\}_{t\in\R}&=W^u(p_1-(0,\pi))\cap W^s(p_1).
\end{aligned}
\end{equation}
\item The ones belonging to $W^u(p_6)\cap W^s(p_7)$.
\item The ones belonging to $W^u(p_7-(0,\pi))\cap W^s(p_7)$.
\item The ones belonging to $W^u(p_6)\cap W^s(p_6+(0,\pi))$.
\item four orbits $\phi_6,\,\phi_7,\phi_8,\,\phi_9$ such that
\begin{equation}
\begin{aligned}
\{\phi_6(t)\}_{t\in\R}&=W^u(p_6)\cap W^s(p_2),\qquad\{\phi_7(t)\}_{t\in\R}=W^u(p_6)\cap W^s(p_4)\\
\{\phi_8(t)\}_{t\in\R}&=W^u(p_1)\cap W^s(p_7),\qquad\{\phi_9(t)\}_{t\in\R}=W^u(p_3)\cap W^s(p_7),
\end{aligned}
\end{equation}
\end{enumerate}
\begin{lemma}
\label{lemma-stability}
The orbit $\phi_6=(u_6,v_6)$ fulfils $u'_6<0$, while the orbit $\phi_7=(u_7,v_7)$ fulfils $u'_7>0$.
\end{lemma}
\begin{proof}
We prove, for example, that $u'_6<0$. We write
$$\bar{D}=\cup_{i=1}^4 \bar{\Omega}_i,$$
where
\begin{equation}\notag
\begin{aligned}
\Omega_1&:=\{(u,v)\in D:\,u<\bar{\alpha},\, u<v<\bar{v}(u)\}\\
\Omega_2&:=\{(u,v)\in D:\,v<\min\{\bar{v}(u),u\}\}\\
\Omega_3&:=\{(u,v)\in D:\,u>\bar{\alpha},\, \bar{v}(u)<v<u\}\\
\Omega_4&:=\{(u,v)\in D:\,v>\max\{\bar{v}(u),u\}\}
\end{aligned}
\end{equation}
Being $u'_6=X_1(u_6,v_6)$, the claim is proved if we show that $\phi_6(t)\in\Omega_4$ for any $t\in\R$. The eigenvalues of $JX(p_2)$ are $\lambda_1=1$ and $\lambda_2=-(n-1)$, with corresponding eigenvectors, for instance, $\eta_1=(1,-\frac{m-1}{n})$ and $\eta_2=(0,1)$, while the eigenvalues of $JX(p_6)$ are given by
$$\mu_{1,2}=\frac{\sin\bar{\alpha}\cos\bar{\alpha}}{2}(N\pm\sqrt{N^2-8N+8})$$
with corresponding eigenvectors
\begin{equation}\notag
\begin{aligned}
\xi_1&=\left(1,\frac{-2(N-1)}{3N-2+\sqrt{N^2-8N+8}}\right)\\
\xi_2&=\left(1,\frac{-2(N-1)}{3N-2-\sqrt{N^2-8N+8}}\right).
\end{aligned}
\end{equation}
Computing $\frac{d\bar{v}}{du}(0)$ and $\frac{d\bar{v}}{du}(\bar{\alpha})$ it is possible to see that $\eta_1$ and $-\xi_1$ point towards $\Omega_4$. Since, by the Grobman-Hartman theorem, $\phi_6$ starts from $p_6$ with the direction of $-\xi_1$ and arrive at $p_2$ with the direction of $-\eta_1$, then there exists $T>0$ such that $\phi_6(t)\in \Omega_4$ for $|t|>T$. It remains to prove that $\phi_6(t)\in\Omega_4$ for any $t\in[-T,T]$. If we assume by contradiction that this is not true, then there exist $t_0$ such that $\phi_6$ touches the graph of $\bar{v}$ for the first time. By the phase plane study, it is possible to see that, if it is the case, then $\{\phi_6(t)\}_{t\in\R}$ intersects all the $\Omega_i$, $1\le i\le 4$. In particular, it intersects all the orbits belonging to $W^u(p_6)\cap W^s(p_7)$, which is impossible.
\end{proof}
\begin{remark}
Using that $$u'=\frac{b'a-a'b}{a^2+b^2}$$
and 
\begin{equation}\label{param_O(m)O(n)_invariantnormalvec}
\nu_\Sigma(\pe)=\left(-b'(s){\tt x},a'(s){\tt y}\right)
\end{equation}
so that $\langle \pe,\nu_\Sigma(\pe)\rangle=-ab'+ba'$, it is possible to see that, if $u'$ has constant sign, then the corresponding orbit gives rise to a surface fulfilling property (\ref{stab+}) or (\ref{stab-}). In particular, $\phi_6$ corresponds to $\Sigma^+_{m,n}$, while $\phi_7$ corresponds to $\Sigma^-_{m,n}$.
\end{remark}



\subsection{The Jacobi equation of $\Sigma^\pm_{m,n}$}
In this subsection, we will describe the $O(m)\times O(n)$-invariant Jacobi fields of $\Sigma^\pm_{m,n}$. Since $\Sigma^+_{m,n}=\sigma^{-1}(\Sigma^-_{n,m})$, being $\sigma(x,y):=(y,x)$ the reflection which changes $x$ with $y$, for $(x,y)\in \R^m\times\R^n$, we will only consider the case $\Sigma:=\Sigma^-_{m,n}$. \\

Since we are only interested in $O(m)\times O(n)$-invariant Jacobi fields, that is functions satisfying ${\tt v}(\pe)=v(s)$, $s=s(\pe)$, the equation 
$$\Delta_\Sigma {\tt v}+|A_\Sigma|^2{\tt v}={\tt f}\qquad\text{in $\Sigma$}$$
reads
\begin{equation}
\label{eq-Jacobi-s}
\partial^2_s v+\alpha(s)\partial_s v+\beta(s)=f(s),\qquad\forall\, s\in\R,
\end{equation}
where 
\begin{equation}\label{def:Asigma}
\alpha(s):=(m-1)\frac{a'}{a}+(n-1)\frac{b'}{b},\qquad\beta(s):=|A_\Sigma(\pe)|^2,\qquad f(s):={\tt f}(\pe)
\end{equation}
We introduce the \textit{Emden-Fowler} change of variables $s=e^t$ and we set
\begin{equation}\label{weightsJTEmdenFowler}
\tilde{\alpha}(t):=\alpha(e^t)e^t-1 \quad \hbox{and}\quad \tilde{\beta}(t):= \beta(e^t)e^{2t}.
\end{equation}
and
\begin{equation}\label{Cancellingfirstderterm1}
p(t):=\exp \bigg ({-\int_0^t \frac{\tilde{\alpha}(\tau)}{2}d\tau}\bigg) \quad \hbox{for} \quad t\in\R,
\end{equation}
so that $p(t)$ solves
\begin{equation}\label{Cancellingfirstderterm2}
2\frac{\partial_t p}{p}+\tilde{\alpha}(t)=0.
\end{equation}
We look for solutions $v$ to (\ref{eq-Jacobi-s}) of the form $v(s)=p(t)u(t)$, so that $u$ has to satisfy the equation
\begin{equation}
\partial^2_t u+V(t)u=\tilde{f},
\end{equation}
where
\begin{equation}\label{righthandandpotential}
\begin{aligned}
\tilde{f}(t)&:=\frac{e^{2t}}{p(t)}f(e^t)\quad \hbox{and} \quad 
V(t)&:=\frac{\partial^2_t p}{p}+\tilde{\alpha}(t)\frac{\partial_t p}{p}+\tilde{\beta}(t)
\end{aligned}
\end{equation}
for $t\in \R$. These computations were made in Subsection $2.3$ of \cite{AKR1} too.

We need the following non degeneracy result.
\begin{proposition}\label{prop_Jacobi_Sigma}
Let $m,n\ge 2$, $m+n\ge 8$. Then $\Sigma:=\Sigma^{-}_{m,n}$ has exactly two $O(m)\times O(n)$-invariant linearly independent Jacobi fields ${\tt v}_\pm(\pe)=v_\pm(s(\pe))>0$ such that
\begin{itemize}
\item[(i)] $v_+(s)$ is smooth, even in the variable $s$, $v_+(0)=1$ and
\begin{equation}\label{JacField+asympt}
v_+(s)=c_1 s^{\gamma} \big( 1 + \mathcal{O}(s^{-\alpha})\big)\quad  \hbox{as} \quad s\to\infty
\end{equation}
with $\gamma\in\{\gamma_\pm\}$.
\medskip
\item[(ii)] $v_-(s)$ is smooth except at $s=0$, where it is singular if $n>2$ and for some $c_2>0$
\begin{equation}\label{JacField-asympt1}
v_-(s)= \left\{
\begin{aligned}
s^{-(n-2)}(1 + \mathcal{O}(s^2)) & \quad \hbox{as} \quad s \to 0\\
c_2s^{\bar{\gamma}} \big( 1 + \mathcal{O}(s^{-\alpha})\big) &\quad  \hbox{as} \quad s \to \infty,
\end{aligned}
\right.
\end{equation}
where $\alpha\in (0,1)$ and $\bar{\gamma}\in\{\gamma_{\pm}\}$, $\bar{\gamma}\ne\gamma$. Moreover, relation \eqref{JacField+asympt} and \eqref{JacField-asympt1} can be differentiated.
\end{itemize} 
\end{proposition}
\begin{proof}
Point (i) follows setting $${\tt v}_+(\pe):=-\langle \pe,\nu_\Sigma(\pe)\rangle=ab'-a'b>0,$$ 
which is positive due to Theorem \ref{Th_Sigma_detailed}, and point (1) of Remark \ref{rem-Sigma}. Point (ii) follows by taking an Emden-Fowler change of variables and using the variation of parameters formula, which is possible thanks to the fact that ${\tt v}_+>0$ and the asymptotic behaviour of ${\tt v}_+$ as $s\to 0$ and as $s\to\infty$. See Proposition $2.1$ of \cite{AKR1} for the explicit computations, they are exactly the same. The only difference in the case $n=2$ or $m=2$ is that we do not know which of the two Jacobi fields has asymptotic behaviour given by $\gamma_+$ and which one has asymptotic behaviour given by $\gamma_-$.
\end{proof}

\section{Jacobi-Toda system}
In this section we prove Theorem (\ref{main-th-JT}). We look for an $O(m)\times O(n)$-invariant solution to (\ref{JT-system}) of the form
$$\h={\tt v}+\q,$$
where ${\tt v}:=({\tt v}_1,\dots,{\tt v}_k)$ is a suitable  \textit{approximate solution} and $\q:=(\q_1,\dots,\q_k)$ is a {\it small correction} with respect to $\delta$ is some suitable topology (see subsection \ref{subs-decoupling}).

Using the convention that 
$$
{\tt v}_0=\q_0=-\infty,\qquad{\tt v}_{k+1}=\q_{k+1}=\infty,$$
for $C^2$ functions ${\tt v}=({\tt v}_1,\dots,{\tt v}_k):\Sigma\to\R^k$ and $\q=(\q_1,\dots,\q_k):\Sigma\to\R^k$ we introduce the notations
\begin{equation}
\label{errorJacToda}
\begin{aligned}
\ta_0=\ta_k&:=0,\qquad \ta_j:=e^{-\sqrt{2}({\tt v}_{j+1}-{\tt v}_j)},\, 1\le j\le k-1\\
E_{\delta,j}({\tt v})&:=\delta J_\Sigma {\tt v}_j-\gamma (e^{-\sqrt{2}({\tt v}_j-{\tt v}_{j-1})}-e^{-\sqrt{2}({\tt v}_{j+1}-{\tt v}_j)})\\
\mathcal{L}_{\delta,j}(\q)&:=\delta J_\Sigma \q_j+\gamma\sqrt{2} (\ta_{j-1}(\q_j-\q_{j-1})-\ta_j(\q_{j+1}-\q_j)).
\end{aligned}
\end{equation}
Setting $Q(t):=e^{-\sqrt{2}t}-1+\sqrt{2}t$, $Q(\infty)=0$,  system (\ref{JT-system}) can be rewritten in the form
\begin{equation}
\label{syst-JT-sol-appr}
E_{\delta,j}({\tt v})+\mathcal{L}_{\delta,j}(\q)-\gamma \ta_{j-1}Q(\q_j-\q_{j-1})+\gamma \ta_j Q(\q_{j+1}-\q_j)=0,\qquad 1\le j\le k.
\end{equation} 

\subsection{The approximate solution}\label{subs-approx-sol}

Let $W:[0,\infty)\to [0,\infty)$ be the \textit{Lambert function}, which is the inverse of the function $te^t\in[0,\infty)\mapsto [0,\infty)$. In other words, for any given $z \ge 0$, $W(z)$ uniquely solves
\begin{equation}
W(z)e^{W(z)}=z,\qquad\forall z\ge 0
\label{def-Lambert}
\end{equation}

Clearly, $W$ is smooth, $W(z)>0$ for $z>0$ and it is well known that
\begin{equation}\label{asympLambertfnzeroinfty}
W(z)= \left\{
\begin{aligned}
z-z^2+& \mathcal{O}(z^4), \quad & \hbox{as} \quad &z\to 0^+,\\
\log(z)- \log \big(\log (z)\big)&+ \mathcal{O}\bigg(\frac{\log \big(\log (z)\big)}{\log (z)}\bigg), \quad & \hbox{as} \quad  &z\to \infty
\end{aligned}
\right.
\end{equation}
and these relations can be differentiated. From (\ref{asympLambertfnzeroinfty}), for any $i\geq 1$, there exists $C_i$ such that 
 $$|W^{(i)}(z)|\le\frac{C_i}{(1+z)^i},\qquad\forall\, z\ge 0 \quad \hbox{(see Lemma $3.1$ of \cite{AKR1})}.
$$

Recall that $\beta(s)=|A_\Sigma(\pe)|^2$, where $s=s(\pe)\in\R$ is the arc length along the profile curve $(a,b)$ which defines $\Sigma$. We set 
\begin{equation}
\label{def-w}
w(s):=\frac{1}{\sqrt{2}}W\left(\frac{\gamma\sqrt{2}}{\delta\beta(s)}\right).
\end{equation}
 and observe that
\begin{equation}
\label{eq-w}
\delta \beta(s)w=\gamma e^{-\sqrt{2}w}
\end{equation}
and
\begin{equation}
\label{no-intersection-0}
w(s)=O(\log(s^2+2)+|\log\delta|)
\end{equation}
as $\delta\to 0$, uniformly in $s\ge 0$ with
\begin{equation}
\label{est-der-w}
|w^{(i)}(s)|\le\frac{C_i}{(1+s)^i},\qquad\forall\,i\ge 1,\, s\ge 0.
\end{equation}

We use the function $w$ to construct the approximate solution ${\tt v}$ to our Jacobi-Toda system (\ref{JT-system}), whose main term will be given by
\begin{equation}
\label{def-w_0}
{\tt w}^0_j(\pe):=\left(j-\frac{k+1}{2}\right)w(s).
\end{equation}
where $s=s(\pe)$.

\begin{lemma}\label{lemmaapproximateslnJacToda}
For any $\delta>0$ and for any $l\ge 1$, there exist $O(m)\times O(n)-$invariant functions ${\tt w}^0,\dots,{\tt w}^l:\Sigma\to\R^k$ defined on $\Sigma$ which are smooth and satisfy the following:
\begin{itemize}
\item[i.] \begin{equation}
\label{w-sum-0}
\sum_{j=1}^k {\tt w}^l_j=0,\qquad\forall\, l\ge 1;
\end{equation}

\item[ii.] for every $1\le j\le k,\,l\ge 1,\, i\ge 0$
\begin{equation}
\label{est_approx_sol}
|\partial_s^{(i)}w^l_j|\le \frac{C_l}{(s+1)^i(\log(s^2+2))^{\frac{l-1}{2}}|\log\delta|^{\frac{l-1}{2}}},\qquad\forall\, s\ge 0,
\end{equation}
where $s=s(\pe)$ and ${\tt w}^l(\pe)=w^l(s)$;

\item[iii.] the error defined in  \eqref{errorJacToda} related to the function
\begin{equation}\label{recursiveApproxslnJacToda}
{\tt v}^l:={\tt w}^0+\dots+{\tt w}^l
\end{equation}
is given by
\begin{equation}
\label{err_j}
E_{\delta,j}({\tt v}^l)=\delta J_\Sigma {\tt w}^l_j \quad \hbox{in} \quad \Sigma,\,\forall\, 1\le j\le k,\,l\ge 1
\end{equation}
in the coordinate $s=s(\pe)$;

\item[iv.] \begin{equation}
\label{est_err_j}
|E_{\delta,j}(v^l)|\le \frac{C_l\delta }{(1+s)^2(\log(s^2+2))^{\frac{l-1}{2}}|\log\delta|^{\frac{l-1}{2}}},\qquad\forall\,s\ge 0,\, 1\le j\le k,\,l\ge 1
\end{equation}
for some constant $C_l>0$ depending only on $l$.
\end{itemize}
\end{lemma}

\begin{remark}
Thanks to (\ref{asympLambertfnzeroinfty}),  (\ref{def-w_0}) and (\ref{est_approx_sol}),  for $\delta>0$ sufficiently small  the functions ${\tt v}_1,\ldots, {\tt v}_k $ satisfy (\ref{disj-graph}), i.e., their graphs are mutually disjoint.
\end{remark}

\begin{remark}
\label{rem-matrix}
Let $k\ge 2$ be an integer and let $b_1,\dots,b_{k-1}\in(0,\infty)$. Then the $k\times(k-1)$-matrix
$$M:=\begin{bmatrix}
-b_1 & \dots&0\\
b_1 & -b_2 &\dots\\
\dots & b_{k-2} & -b_{k-1}\\
0 &\dots & b_{k-1} 
\end{bmatrix}$$
has trivial kernel, since the columns are linearly independent.  On the other hand, for any $a:=(a_1,\dots,a_{k-1})\in\R^{k-1}$, we have
$$\langle M a,e\rangle=-b_1 a_1+b_1 a_1+\dots-b_{k-1}a_{k-1}+b_{k-1}a_{k-1}=0$$
and 
$$
{\rm rank}(M)=k-1-{\rm dim}(Ker M)=k-1.$$ and hence the range of $M$ is given by $span\{e\}^\bot\subset\R^k$, where $e:=(1,\dots,1)\in \R^k$.

As a consequence, for any $c\in\R^k$ such that $\langle c,e\rangle=0$, there exists a unique $a\in\R^{k-1}$ such that $$Ma=c.$$ 
By recursion, the solution $a\in\R^{k-1}$ is explicitly determined by the formula
$$a_j=-\frac{\sum_{i=1}^j c_i}{b_j},\qquad 1\le j\le k-1.$$ 
\end{remark}
Now we can prove Lemma \ref{lemmaapproximateslnJacToda}.
\begin{proof}
The proof is by recursion on $l$. Since all the functions we deal with are $O(m)\times O(n)$-invariant, we will use the notation  ${\tt w}^l(\pe)=w^l(s)$ and ${\tt v}^l(\pe)=v^l(s)$, where $s=s(\pe)$. Similarly, with a slight abuse of notation, we write $J_\Sigma w^l_j=J_\Sigma {\tt w}^l_j$ and $\Delta_\Sigma w^l_j=\Delta_\Sigma {\tt w}^l_j$.

We start with the case $l=1$. From (\ref{def-w_0}), we set 
$$
v^1_j(s):=w^0_j(s)+ w^1_j(s)=\left(j-\frac{k+1}{2}\right)w(s) + w^1_j(s),\qquad 1\le j\le k
$$
and from (\ref{eq-w}) we compute
\begin{equation}\notag
\begin{aligned}
E_{\delta,j}(v^1)&=\delta\left(j-\frac{k+1}{2}\right) J_\Sigma w+\delta J_\Sigma w^1_j-\gamma e^{-\sqrt{2}w}(e^{-\sqrt{2}(v^1_j-v^1_{j-1})}-e^{-\sqrt{2}(v^1_{j+1}-v^1_j)})=\\
&=\delta\left(j-\frac{k+1}{2}\right) J_\Sigma w+\delta J_\Sigma w^1_j-\delta\beta(s)w(e^{-\sqrt{2}(w^1_j-w^1_{j-1})}-e^{-\sqrt{2}(w^1_{j+1}-w^1_j)}),\qquad\forall\, 1\le j\le k.
\end{aligned}
\end{equation}
Now we choose $w^1$ in such a way that
$$E_{\delta,j}(v^1)=\delta J_\Sigma w^1_j,$$
that is
$$\left(j-\frac{k+1}{2}\right) J_\Sigma w-\beta(s)w(e^{-\sqrt{2}(w^1_j-w^1_{j-1})}-e^{-\sqrt{2}(w^1_{j+1}-w^1_j)})=0,\qquad\forall\, 1\le j\le k.$$
In order to do so, setting
\begin{equation}
\label{def-a_j}
a^1_0=a^1_k:=0,\qquad a^1_j:=e^{-\sqrt{2}(w^1_{j+1}-w^1_j)},\qquad\forall\, 1\le j\le k-1,
\end{equation}
we have to solve the algebraic system
$$a^1_{j-1}-a^1_j=\left(j-\frac{k+1}{2}\right)\frac{J_\Sigma w}{\beta(s)w},\qquad\forall\, 1\le j\le k.$$
Applying Remark \ref{rem-matrix} with $b_1=\dots=b_{k-1}=1$, it is possible to compute
\begin{equation}
\label{est-a_j}
a^1_j=\frac{j(k-j)}{2}\frac{J_\Sigma w}{\beta w}=\frac{j(k-j)}{2}\left(1+\frac{\Delta_\Sigma w}{\beta w}\right),\qquad 1\le j\le k-1
\end{equation}
so that
$$w^1_{j+1}-w^1_j=-\frac{1}{\sqrt{2}}\log\left(\frac{j(k-j)}{2}\left(1+\frac{\Delta_\Sigma w}{\beta w}\right)\right),\qquad 1\le j\le k-1.
$$

Imposing the condition
$$
\sum_{j=1}^k w^1_j=0
$$
we solve a linear algebraic system for the function $w^1=(w^1_1,\dots,w^1_k)$. Using (\ref{asympLambertfnzeroinfty}) and (\ref{est-der-w}) we see that
$$|\partial^{(i)}w^1_j|\le\frac{C_i}{(1+s)^i}, \qquad\forall\,1\le j\le k,\, s\ge 0,\, i\ge 0
$$
and, in particular,
$$
|E_{\delta,j}(v^1)|=\delta|J_\Sigma w^1_j|\le \frac{C_1 \delta}{(s+1)^2},\qquad\forall\,1\le j\le k
$$
so that (\ref{est_approx_sol}) holds.\\

In order to conclude the proof, we need to prove the step $l>1$. First we note that
$$E_{\delta,j}(v^{l})=E_{\delta,j}(v^{l-1})+\delta J_\Sigma w^l_j- e^{-\sqrt{2}(v^{l-1}_j-v^{l-1}_{j-1})}(e^{-\sqrt{2}(w^l_j-w^l_{j-1})}-1)+
e^{-\sqrt{2}(v^{l-1}_{j+1}-v^{l-1}_j)}(e^{-\sqrt{2}(w^l_{j+1}-w^l_j)}-1)$$
for $1\le j\le k$. We want to get
$$E_{\delta,j}(v^l)=\delta J_\Sigma w^l_j.$$
In order to do so, setting
$$a^l_0=a^l_k=b^l_0=b^l_k=0,\qquad a^l_j:=e^{-\sqrt{2}(w^l_{j+1}-w^l_j)}-1,\qquad b^l_{j}:=e^{-\sqrt{2}(v^{l-1}_{j+1}-v^{l-1}_j)}, \qquad 1\le j\le k-1,\, l\ge 2$$
we have to solve the system
\begin{equation}
\label{syst-a^l}
b^l_{j-1} a^l_{j-1}-b^l_{j} a^l_j=E_{\delta,j}(v^{l-1}), \qquad 1\le j\le k-1.
\end{equation}
with respect to $a^l:=(a^l_1,\dots,a^l_{k-1})$. This is possible thanks to Remark \ref{rem-matrix} applied with $b_j:=b^l_j>0$, $1\le j\le k-1$. To justify this statemente, observe that the right hand side
$$E_{\delta,j}(v^{l-1})=\delta J_\Sigma w^{l-1}_j,$$
satisfies the required orthogonality condition. Indeed, by the inductive hypothesis 
$$\sum_{j=1}^k w^{l-1}_j=0,
$$
and hence
$$\sum_{j=1}^k E_{\delta,j}(v^{l-1})=\sum_{j=1}^k\delta J_\Sigma w^{l-1}_j=\delta J_\Sigma\Big(\sum_{j=1}^k w^{l-1}_j\Big)=0.
$$

As a consequence, the function $w^l=(w^l_1,\dots,w^l_k)$ is uniquely determined by the algebraic system
\begin{equation}
\label{def-w^l}
\begin{aligned}
w^l_{j+1}-w^l_j&=-\frac{1}{\sqrt{2}}\log(1+a^l_j),\qquad 1\le j\le k-1\\
\sum_{j=1}^k w^l_j&=0
\end{aligned}
\end{equation}
Estimates (\ref{est_approx_sol}) and (\ref{est_err_j}) follow from Remark \ref{rem-matrix}, the inductive hypothesis, which gives
$$|\partial_s^{(i)}w^{l-1}_j|\le \frac{C_{l-1}\delta}{(s+1)^2\log(s^2+2)^{\frac{l-2}{2}}|\log\delta|^{\frac{l-2}{2}}},\qquad \forall \,s\ge 0,\, 1\le j\le k,\, i\ge 0,$$
the fact that
\begin{equation}\notag
\begin{aligned}
b^{l}_j&=e^{-\sqrt{2}w}e^{-\sqrt{2}(w^{1}_{j+1}-w^{1}_j)}(1+o(1))=\\
&=\gamma^{-1}\delta \beta(s)we^{-\sqrt{2}(w^{1}_{j+1}-w^{1}_j)}(1+o(1))=wO\left(\frac{C_l\delta}{(s+1)^2}\right),\qquad 1\le j\le k-1
\end{aligned}
\end{equation}
and differentiating relations (\ref{def-w^l}) and (\ref{syst-a^l}). This completes the proof of the lemma.
\end{proof}

\subsection{A decoupling argument}\label{subs-decoupling}

In this section we consider system (\ref{syst-JT-sol-appr}) where ${\tt v}:={\tt v}^l$ is the approximate solution constructed in Subsection \ref{subs-approx-sol}. Let $l>1$  be arbitrary, but fixed, and to be specified later in Subsection \ref{subs-fixed-point-JT}.

Next, motivated by the behaviour of the error $E_{\delta,j}({\tt v})$ described in \eqref{est_err_j}, we introduce the suitable function spaces for our framework.

For $\beta \in (0,1)$, $\mu>0$, $\varrho\in\R$ and a function ${\tt f}:\Sigma\to\R$, we introduce the norms
\begin{equation}\label{decayingLinftynorm}
\|{\tt f}\|_{*,\mu,\varrho}:=\sup_{\pe\in\Sigma}(s(\pe)^2+2)^{\frac{\mu}{2}} \log(s(\pe)+2)^\varrho\|{\tt f}\|_{L^\infty(B_1(\pe))}.
\end{equation}

\begin{equation}\label{decayingHoldernorm}
\|{\tt f}\|_{\mathcal{D}^{0,\beta}_{\mu,\varrho}(\Sigma)}:=\sup_{\pe\in\Sigma}(s(\pe)^2+2)^{\frac{\mu}{2}}\log(s(\pe)+2)^\varrho\|{\tt f}\|_{C^{0,\beta}(B_1(\pe))}
\end{equation}
and we consider the Banach space $\mathcal{D}^{0,\beta}_{\mu,\varrho}(\Sigma)$ defined as the space of $O(m)\times O(n)$-invariant functions ${\tt f}\in C^{0,\beta}_{loc}(\Sigma)$ for which the norm
\begin{equation}\label{functionspacerighthandsideJT}
\|{\tt f}\|_{\mathcal{D}^{0,\beta}_{\mu,\varrho}(\Sigma)}< \infty.
\end{equation}

We also consider the Banach space $\mathcal{D}^{2,\beta}_{\mu,\varrho}(\Sigma)$ defined as the space of $O(m)\times O(n)-$invariant functions ${\tt q}\in C^{2,\beta}_{loc}(\Sigma)$ for which the norm
\begin{equation}\label{eq_norm_q(s)}
\|{\tt q}\|_{\mathcal{D}^{2,\beta}_{\mu,\varrho}(\Sigma)}:=\sigma^{-1}\|D^2_\Sigma{\tt q}\|_{\mathcal{D}^{0,\beta}_{\mu+2,\varrho-1}(\Sigma)}+\|\nabla_{\Sigma}{\tt q}\|_{*,\mu +1,\varrho}+\|{\tt q}\|_{*,\mu,\varrho}<\infty.
\end{equation}

We note that system (\ref{syst-JT-sol-appr}) can be rewritten as
\begin{equation}
\label{syst_lin_JT_Sigma}
\mathcal{L}_{\delta,j}(\q)=\delta\mathcal{F}_{\delta,j}(\cdotp,\q)
\end{equation}
where 
\begin{equation}
\label{def-remainder-JT}
\delta\mathcal{F}_{\delta,j}(\cdotp,\q):=-E_{\delta,j}({\tt v})+\gamma \ta_{j-1}Q(\q_j-\q_{j-1})-\gamma \ta_j Q(\q_{j+1}-\q_j),\qquad 1\le j\le k,
\end{equation}
the functions ${\ta_j}$ are defined in \eqref{errorJacToda}, $Q$ is defined in \eqref{syst-JT-sol-appr} and $\q:=(\q_1,\dots,\q_k):\Sigma\to\R^k$ is a vectorial $O(m)\times O(n)$-invariant function, with $\q_j\in\mathcal{D}^{2,\beta}_{0,\varrho-\frac{1}{2}}(\Sigma)$, with $\varrho>0$ to be chosen later depending only on $l$.

Setting
\begin{equation}\notag
\begin{aligned}
\hat{\q}_j:&=\q_{j+1}-\q_j\qquad 1\le j\le k-1\\
\hat{\q}_k:&=\sum_{j=1}^k \q_j
\end{aligned}
\end{equation}
or equivalently $\hat{\q}:=B\q$ where $B$ is the $k\times k$ matrix given by
\begin{equation}
\label{def-B}
B:=\begin{bmatrix}
-1 & 1 &\dots &0\\
0 & -1 & 1 &\red{\vdots}\\
0 & \dots & -1 & 1\\
1 &\dots & \dots & 1
\end{bmatrix},
\end{equation}
the vector valued function $\hat{\q}$ must satisfy the system
\begin{equation}
\label{syst-hat}
\begin{aligned}
\delta J_\Sigma \hat{\q}_j+\gamma\sqrt{2}(-\ta_{j-1}\hat{\q}_{j-1}+2\ta_j\hat{\q}_j-\ta_{j+1}\hat{\q}_{j+1})&=
\delta\hat{\mathcal{F}}_{\delta,j}(\cdotp,B^{-1}\hat{\q})\qquad 1\le j\le k-1\\
J_\Sigma \hat{\q}_k&= 0
\end{aligned}
\end{equation}
in $\Sigma$, where $\hat{\mathcal{F}}_{\delta}(\cdotp,\q):=B \mathcal{F}_{\delta}(\cdotp,\q)$. We note that $\hat{\mathcal{F}}_{\delta,k}(\cdotp,\q)=0$, so that we can choose $\hat{\q}_k=0$. Setting $\q^\sharp:=(\hat{\q}_1,\dots,\hat{\q}_{k-1}):\Sigma\to\R^{k-1}$ and defining the $(k-1)\times(k-1)$ matrix
\begin{equation}
\label{def-A}
\mathcal{A}:=\begin{bmatrix}
2\ta_1 & -\ta_2 & \dots & 0 \\
-\ta_1 & 2\ta_2 & \ta_3 &\dots\\
\dots & -\ta_{k-3} & 2\ta_{k-2} & -\ta_{k-1}\\
0 &\dots & -\ta_{k-2} & 2\ta_{k-1}
\end{bmatrix},
\end{equation}
then it is possible to rewrite (\ref{syst-hat}) as
\begin{equation}
\label{syst-sharp}
\delta J_\Sigma \q^\sharp+\gamma\sqrt{2}\mathcal{A}\q^\sharp=\delta\mathcal{F}^\sharp_{\delta}(\cdotp,B^{-1}(\q^\sharp,0))
\end{equation}
where $\mathcal{F}^\sharp_{\delta}(\cdotp,\q):=(\hat{\mathcal{F}}_{\delta,1}(\cdotp,\q),\dots,\hat{\mathcal{F}}_{\delta,k-1}(\cdotp,\q))$.\\

For future purposes we write $\mathcal{A}=CD$, where $C$ and $D$ are $(k-1)\times (k-1)$ matrices defined by
$$C:=\begin{bmatrix}
2 & -1 \dots &\dots\\
-1 & 2 & -1 &\dots\\
\dots & -1 & 2 & -1\\
\dots &\dots & -1 & 2
\end{bmatrix},$$
and $D:=diag(\ta_j)_{1\le j\le k-1}$.
\begin{lemma}
\label{lemma-A}
Using the $s$ variable, that is setting $a_j(s):=\ta_j(\pe)$, we have
$$a_0=a_k=0,\qquad a_j=\delta\gamma^{-1}\beta(s)w(s)\left(\frac{j(k-j)}{2}
+O\left(\frac{1}{\log(s^2+2)^{\frac{1}{2}}|\log\delta|^{\frac{1}{2}}}\right)\right),\qquad 1\le j\le k-1.$$
In particular, $a_j>0$ for $1\le j\le k-1$ and $\delta>0$ small enough.
\end{lemma}
\begin{proof}
We note that, using definition of $\ta_j$ given in (\ref{errorJacToda}), the definition of the approximate solution ${\tt v}$ given by (\ref{def-w_0}) and (\ref{recursiveApproxslnJacToda}),  and the definition of $w$ given by (\ref{def-w}) we have
\begin{equation}
\begin{aligned}
a_j&=e^{-\sqrt{2}(v_{j+1}-v_{j})}=e^{-\sqrt{2}w}e^{-\sqrt{2}(w^1_{j+1}-w^1_{j})}e^{-\sqrt{2}\sum_{i=2}^l(w^i_{j+1}-w^i_{j})}\\
&=\delta\gamma^{-1} \beta(s) w e^{-\sqrt{2}(w^1_{j+1}-w^1_{j})}e^{-\sqrt{2}\sum_{i=2}^l(w^i_{j+1}-w^i_{j})}
\end{aligned}
\end{equation}
Recalling relations (\ref{def-a_j}) and (\ref{est-a_j}), we can see that
$$e^{-\sqrt{2}(w^1_{j+1}-w^1_{j})}=\frac{j(k-j)}{2}\left(1+\frac{\Delta_\Sigma w}{\beta(s)w}\right)=\frac{j(k-j)}{2}\left(1+O\left(\frac{1}{\log(s^2+2)^{\frac{1}{2}}|\log\delta|^{\frac{1}{2}}}\right)\right),$$
thanks to the asymptotic behaviour of $w$, given by (\ref{asympLambertfnzeroinfty}) and (\ref{est-der-w}). Moreover, by (\ref{est_approx_sol}) we have
\begin{equation}\notag
e^{-\sqrt{2}\sum_{i=2}^l(w^i_{j+1}-w^i_j)}=1+O\left(\frac{1}{\log(s^2+2)^{\frac{1}{2}}|\log\delta|^{\frac{1}{2}}}\right),\qquad 1\le j\le k-1,
\end{equation}
which concludes the proof.
\end{proof}
We note that the constant matrix $\mathcal{A}_0:=C\,diag(\frac{j(k-j)}{2})_{1\le j\le k-1}$, has $k-1$ eigenvalues $\mu_1^0,\dots,\mu^0_{k-1}\in (0,\infty)$. Therefore, in view of Lemma \ref{lemma-A}, the eigenvalues of $\mathcal{A}$ can be expanded as
\begin{equation}
\mu_j=\delta\gamma^{-1}|A_\Sigma(\pe)|^2\w(\pe)(\mu^0_j+{\tt r}_j),\qquad \forall\, 1\le j\le k-1
\end{equation}
where ${\tt r}_j:\Sigma\to\R$ are smooth $O(m)\times O(n)$-invariant functions such that
$$|{\tt r}_j(\pe)|=|r_j(s)|\le \frac{C}{\log(s^2+2)^{\frac{1}{2}}|\log\delta|^{\frac{1}{2}}}.$$
In particular, $\mu_j>0$ for $\delta>0$ sufficiently small, and $\mathcal{A}$ is similar to $M:=diag(\mu_j)_{1\le j\le k}$, in the sense that there exists a  invertible matrix $V$, depending on $\delta$ and $s(\pe)$, with smooth and bounded, $O(m)\times O(n)$-invariant entries and such that $A=V M 
V^{-1}$.

Setting $\q^\sharp=V\tilde{\q}$, system (\ref{syst-sharp}) is equivalent to
\begin{equation}
\label{syst-tilde-q}
\Delta_\Sigma \tilde{\q}_j+|A_\Sigma|^2(1+\sqrt{2}(\mu^0_j+{\tt r}_j) \w)\tilde{\q}_j
=\mathcal{G}_{\delta,j}(\tilde{\q}),\qquad\forall \,1\le j\le k-1
\end{equation}
where
\begin{equation}
\label{def-G}
\mathcal{G}_{\delta,j}(\tilde{\q}):=(V^{-1}\mathcal{F}^\sharp_{\delta})_{j}(\cdotp,B^{-1}(V\tilde{\q},0)),\qquad\forall \,1\le j\le k-1
\end{equation}


First we focus on the linear problem
\begin{equation}
\label{eq_lin_JT_Sigma}
\Delta_\Sigma \q+|A_\Sigma|^2(1+\sqrt{2}(\mu^0_j+{\tt r}_j)\w)\q={\tt f}\qquad\text{in $\Sigma$,}
\end{equation}
where $\q,\,{\tt f}:\Sigma\to\R$ are real valued functions.

\medskip
The following proposition shows that \eqref{eq_lin_JT_Sigma} has an inverse and also it accounts for its size respect to $\sigma=\ln\big(\frac{1}{\delta}\big)$, for $\delta>0$ small.

\begin{proposition}\label{propinverselinearJT}
Let $\beta \in (0,1)$ and $\varrho\ne\frac{3}{4}$. There exist $\sigma_0>0$ and a constant $C>0$ such that for any $\sigma>\sigma_0$ and any ${\tt f}\in \mathcal{D}^{0,\beta}_{2,\varrho}(\Sigma)$, equation \eqref{eq_lin_JT_Sigma} has a solution ${\tt q}:={\tt F}_1({\tt f})\in\mathcal{D}^{2,\beta}_{0,\varrho-\frac{1}{2}}(\Sigma)$ satisfying the estimate
\begin{equation}
\label{prop_eq_lin_JT_Sigma}
\|{\tt q}\|_{\mathcal{D}^{2,\beta}_{0,\varrho-\frac{1}{2}}(\Sigma)}\le  c\sigma^{\frac{3}{4}}\log(\sigma)\|{\tt f}\|_{\mathcal{D}^{0,\beta}_{2,\varrho}(\Sigma)}.
\end{equation}
\end{proposition}
In the case $m\ge 2,\,n\ge 3$, Proposition \ref{propinverselinearJT} is proved exactly as Proposition $3.1$ in \cite{AKR1}. The only difference is given by the function multiplying $\w$, but it is totally irrelevant since the asymptotic behaviour as $\delta\to 0$ and $s\to\infty$ is the same as in that case. We give a detailed proof in the case $n=2$ in subsections \ref{subs-est-neg-pot}, \ref{subs_transI} and \ref{subs-large-t}.

Using the symmetries and introducing the Emden-Fowler change of variables $s=e^t$ and setting
\begin{equation}
\begin{aligned}
\tilde{\alpha}(t)&:=\alpha(e^t)e^t-1,\quad\tilde{\beta}(t):=e^{2t}\beta(e^t),\\
\tilde{w}(t)&:=1+\sqrt{2}(\mu^0_j+r_j(e^t))w(e^t),\quad Q(t):=\frac{\partial^2_t p}{p}+\tilde{\alpha}\frac{\partial^2_t p}{p}+\tilde{w}\tilde{\beta}>0,
\end{aligned}
\end{equation}
we reduce ourselves to study an equation of the form
\begin{equation}
\label{eq-Jacobi-Toda-lin-EF}
\partial^2_t v+Q(t)v=\tilde{f}(t)=\frac{e^{2t}}{p(t)}f(e^t).
\end{equation}
Using that
\begin{equation}
\tilde{w}(t)= 
\begin{cases}
\sigma+O(1)\qquad\text{as $t<T_0$}\\
\sigma+2\mu_0^j t+O(\log t)\qquad\text{as $t>T_1$}
\end{cases}
\end{equation}
we get
\begin{equation}
Q(t)=
\begin{cases}
c_0 e^{2t}+O(e^{2t})+O(\sigma e^{4t})\qquad\text{as $t<T_0$}\\
(N-1)(\sigma+2t)+O(\log t)+O(\sigma e^{-t})+O(te^{-t})\qquad\text{as $t>T_1$}
\end{cases}
\end{equation}
and these relations can be differentiated. First we look for two linearly independent solutions $v$ and $\tilde{v}$ to the homogeneous equation
\begin{equation}
\label{eq-Jacobi-Toda-lin-EF-hom}
\partial^2_t v+Q(t)v=0.
\end{equation}
In order to solve equation (\ref{eq-Jacobi-Toda-lin-EF-hom}) we will distinguish three regions, where
\begin{equation}\label{summaryregimesQ}
\left\{
\begin{aligned}
c_0\sigma e^{2t} & \le Q(t) \le C_0 \sigma e^{2t},&\qquad \partial_t Q>0, \qquad  t\le T_0\\
0<  c_1\sigma&\le Q(t) \le   C_1\sigma,& \qquad T_0< t\le T_1\\
c_2(\sigma+t) & \leq Q(t) \leq C_2(\sigma+t), \qquad & t>T_1.
\end{aligned}
\right.
\end{equation}

\subsection{Estimates for negative $t$}\label{subs-est-neg-pot}

We study equation (\ref{eq-Jacobi-Toda-lin-EF}) in $(-\infty,t_\sigma)$, with $t_\sigma<0$ to be chosen later, with $\sigma>0$ large and $v(t)=tx(t)$, so that $x$ has to satisfy the equation
$$\partial_t(t^2\partial_t x)=-Q(t)t^2 x(t),$$
which is equivalent to the system
\begin{equation}
\label{syst-neg-t}
\begin{cases}
\partial_t x=t^{-2}y,\qquad x(t_\sigma)=1\\
\partial_t y=-Q(t)t^2 x,\qquad y(t_\sigma)=0
\end{cases}
\end{equation}
in $(-\infty,t_\sigma)$. Integrating (\ref{syst-neg-t}) and using the initial conditions
\begin{equation}
\label{syst-xy-int}
x(t)=1-\int_t^{t_\sigma}\tau^{-2} y(\tau)d\tau,\qquad y(t)=\int_t^{t_\sigma} Q(\tau)\tau^2 x(\tau)d\tau,
\end{equation}
thus, by Fubini's Theorem,
\begin{equation}\notag
\begin{aligned}
x(t)&=1-\int_t^{t_\sigma} \tau^{-2}\int_\tau^{t_\sigma} Q(\xi)\xi^2 x(\xi)d\xi d\tau\\
&=1-\int_t^{t_\sigma}Q(\xi)\xi^2 x(\xi)\int_t^\xi \tau^{-2}d\tau d\xi,
\end{aligned}
\end{equation}
so that
$$|x(t)|\le 1-\int_t^{t_\sigma}Q(\xi)\xi|x(\xi)|d\xi$$
therefore, by The Gronwall inequality,
\begin{equation}\notag
|x(t)|\le\exp\left(-\int_t^{t_\sigma}Q(\xi)\xi d\xi\right).
\end{equation}
Moreover,
\begin{equation}\notag
\begin{aligned}
-\int_t^{t_\sigma}Q(\xi)\xi d\xi&\le c\sigma\int_t^{t_\sigma}(-\xi)e^{2\xi} d\xi\le \\
&c\sigma\left[-\frac{1}{2}\xi e^{2\xi}\right]^{t_\sigma}_t+c\sigma\int_t^{t_\sigma}\frac{1}{2}e^{2\xi}d\xi<\\
&c\sigma\left(-t_\sigma e^{2t_\sigma}+\frac{1}{2}e^{2t_\sigma}\right)\le ce^{-2M}
\end{aligned}
\end{equation}
provided $t_\sigma<0$ is such that $-2\sigma t_\sigma e^{2t_\sigma}=e^{-2 M}$, with $M>0$ large to be determined later. Studying the inverse of the function $s\mapsto se^{-s}$ for $s>1$, it turns out that 
$$t_\sigma=-\frac{1}{2}\log\sigma(1+o(1)).$$ 
Hence $x\in L^\infty(-\infty,t_\sigma)$ and 
$$\|x\|_{L^\infty(-\infty,t_\sigma)}\le C,$$
for some constant $C>0$. Integrating (\ref{syst-xy-int}), we can see that $y\in L^\infty(-\infty,t_\sigma)$ and
$$\|y\|_{L^\infty(-\infty,t_\sigma)}\le c e^{-2M}(\log\sigma),$$
so that $\partial_t x\in L^\infty(-\infty,t_\sigma)$ and
$$|\partial_t x|\le ce^{-2M}|t|^{-1}.$$
In terms of $v$, this yields that
\begin{equation}
\begin{aligned}
v(t)&=t(1+O(e^{-2M}))<0\\
\partial_t v(t)&=1+O(e^{-2M}).
\end{aligned}
\end{equation}
for $t<t_\sigma$. We stress that
\begin{equation}
v(t_\sigma)=t_\sigma,\qquad \partial_t v(t_\sigma)=1.
\end{equation}
Another linearly independent solutions to (\ref{eq-Jacobi-Toda-lin-EF-hom}) is given by
\begin{equation}\notag
\tilde{v}(t)=v(t)\int_{-\infty}^t (v(\tau))^{-2} d\tau,
\end{equation}
so that
\begin{equation}
\begin{aligned}
\tilde{v}(t)&=-1+O(e^{-2M})\\
\partial_t \tilde{v}(t)&=O(t^{-1}e^{-2M}).
\end{aligned}
\end{equation}
In particular the Wronskian fulfils
\begin{equation}
W(t)\equiv W(t_\sigma)\to 1\qquad\text{as $\sigma\to\infty$.}
\end{equation}
We finished this section by  solving (\ref{eq-Jacobi-Toda-lin-EF}) in $I_1$ via the integral formula
$$
u_1(t):= v_1(t)\int_{-\infty}^{t}W^{-1}(\tau)\tilde{v}_1(\tau)\tilde{f}(\tau)d\tau - \tilde{v}_1(t)\int_{-\infty}^{t}W^{-1}(\tau){v}_1(\tau)\tilde{f}(\tau)d\tau
$$
for $t\in I_0$ from where it is direct to verify that 
\begin{equation}\label{eq:estimatesU0I0}
|u_1(t)|+ |\partial_t u_1(t)| \leq c\|{\tt f}\|_{\mathcal{D}^{0,\beta}_{2,\varrho}(\Sigma)}e^{2t} \quad t\in I_1.
\end{equation}

\subsection{Estimates for the transition region}\label{subs_transI}

In this subsection, we consider the interval $I_2:=(t_\sigma,T_0)$, where the potential $Q(t)$ has the properties that 
$$
0<c_0 <Q(t)\leq c_1\sigma, \qquad  \partial_t Q(t)>0 \quad \hbox{in} \quad I_2 \quad \hbox{see e.g. \eqref{summaryregimesQ}}. 
$$

Observe also that $${\rm lenght}(I_2)=T_0 - t_\sigma=\frac{1}{2}\log(\sigma) + O(1).$$

\medskip

As in the previous subsection, we study the solutions of (\ref{eq-Jacobi-Toda-lin-EF-hom}) in $I_2$. Let $v(t)$ solve (\ref{eq-Jacobi-Toda-lin-EF-hom}) and consider the Lyapunov energy
$$
H(t):=\frac{1}{Q(t)}\frac{\bigl(\partial_t v(t)\bigr)^2}{2} + \frac{v^2(t)}{2} \quad t\in I_2. 
$$

The functional $H(t)$ is differentiable with
$$
\partial_t H(t)=-\frac{\partial_t Q(t)}{Q^2(t)}\frac{\bigl(\partial_t v(t)\bigr)^2}{2} \leq 0 \quad \hbox{in} \quad I_2
$$
and thus 
$$
0 \leq H(t) \leq H(t_{\sigma})\leq c\bigl(|v(t_{\sigma})|^2 +|\partial_t v(t_\sigma)|^2 \bigr) \quad \hbox{for} \quad t\in I_2.
$$ 

\medskip
We estimate for $t\in I_2$
\begin{equation}\label{eq:estimate_generic_kernelI2}
|v(t)| \leq c\bigl(|v(t_{\sigma})| + |\partial_t v(t_\sigma)|\bigr),\qquad |\partial_t v(t)| \leq  c\sqrt{\sigma}\bigl(|v(t_{\sigma})| + |\partial_t v(t_\sigma)|\bigr). 
\end{equation}

Next, we proceed exactly in the same fashion as we did in the previous subsection, first by fixing two solutions $v_2(t)$ and $\tilde{v}_2(t)$ of (\ref{eq-Jacobi-Toda-lin-EF-hom}) in $I_2$, satisfying the initial conditions
$$
\begin{aligned}
v_2(t_{\sigma})=& 1 \quad \partial_{t}v_2(t_{\sigma})=&0\\
\tilde{v}_2(t_{\sigma})=& 0 \quad \partial_{t}\tilde{v}_2(t_{\sigma})=&1
\end{aligned}
$$
with Wronskian determinant $W(t)=1$ in $I_2$. 

\medskip
Next,we set the variations of parameters formula
$$
\begin{aligned}
u_2(t)=& u_1(t_{\sigma})v_2(t) + \partial_t u_1(t_{\sigma})\tilde{v}_2(t)\\
& + v_2(t)\int_{t_{\sigma}}^t \tilde{v}_2(\tau)\tilde{f}(\tau)d\tau - \tilde{v}_2(t)\int_{t_{\sigma}}^t {v}_2(\tau)\tilde{f}(\tau)d\tau 
\end{aligned}$$
for $t\in I_2$.

\medskip
Again, proceeding similarly as 
we did for estimate (\ref{eq:estimatesU0I0}), we obtain
\begin{equation}\label{eq:estimatesU2I2}
|u_2(t)|+ |\partial_t u_2(t)| \leq c\sigma^{\frac{1}{2}}\log(\sigma)\|{\tt f}\|_{\mathcal{D}^{0,\beta}_{2,\varrho}(\Sigma)}e^{2t} \quad t\in I_2.
\end{equation}

\subsection{Estimates for large $t$}\label{subs-large-t}
Now we will study the behaviour of the solutions $v$, $\tilde{v}$ in the interval $(T_0,\infty)$. In this interval, we have $Q\ge 1$ and hence introduce the change of variables
$$
\xi(t)=\int_{T_0}^t Q(\tau)^{\frac{1}{2}}d\tau, \qquad\forall t\ge T_0.
$$

Write
\begin{equation}
\label{def_w}
v(t)=Q(t)^{-\frac{1}{4}}w(\xi(t)) \quad \hbox{for} \quad t \geq T_0 
\end{equation}
and observe from (\ref{eq-Jacobi-Toda-lin-EF-hom}) that $w$ must solve
\begin{equation}
\label{eq_w}
\partial^2_\xi w+(1+{\tt V}(\xi))w=0 \quad \hbox{in} \quad (0,\infty),
\end{equation}
where
$$
{\tt V}(\xi):=-\frac{\partial^2_t Q(t(\xi))}{4Q^2(t(\xi))}+\frac{5(\partial_t Q)^2(t(\xi))}{16 Q^3(t(\xi))}.
$$

Since $\partial_t Q$ and $\partial^2_t Q$ are bounded in $[T_0,\infty)$, we find from (\ref{summaryregimesQ}) that ${\tt V} \in L^1(0,\infty)$ and
\begin{equation}\label{V_L1-norm}
\begin{aligned}
\int_0^\infty |{\tt V}(\xi)|d\xi &=\int_{T_0}^\infty |{\tt V} (\xi(t))|Q^{\frac{1}{2}}(t)dt\\
&\le \int_{T_0}^{T_1}cdt+\int_{T_1}^\infty\frac{c}{(\sigma+t)^{\frac{3}{2}}}dt\\
&\le {\emph{K}}
\end{aligned}
\end{equation}

Moreover, differentiating (\ref{def_w}) and evaluating at $t=T_0$, we get
$$
w(0)=Q(T_0)^{\frac{1}{4}}v(T_0) \quad\hbox{and} \quad \partial_\xi w(0)=\frac{\partial_t Q(T_0)}{4Q(T_0)^{\frac{5}{4}}}v(T_0)-\frac{\partial_t v(T_\sigma)}{Q(T_0)^{\frac{1}{4}}},
$$
which together with (\ref{summaryregimesQ}) yield that
\begin{equation}
\label{w_initial}
\begin{aligned}
|w(0)|+|\partial_\xi w(0)|& \le C\sigma^{\frac{1}{4}}(|v(T_0)|+|\partial_t v(T_0)|).
\end{aligned}
\end{equation}

The aim now is to estimate $w$ and $\partial_\xi w$ in the interval $(T_0,\infty)$. In order to do so, we multiply \eqref{eq_w} by $\partial_{\xi} w$ to find that 
$$
\partial_{\xi}\big( \partial_{\xi } w\big)^2 + \partial_{\xi }\big( w\big)^2  = -2 {\tt V}(\xi)w \partial_{\xi}w 
$$
and a direct integration yields that 
$$
\begin{aligned}
\big( \partial_{\xi } w(\xi)\big)^2 + \big( w(\xi)\big)^2  &= \int_{0}^{\xi}-2 {\tt V}(z)w \partial_{z}wdz\\
& \leq \int_{0}^{\xi}|{\tt V}(z)|\bigg(\big( \partial_{\xi} w\big)^2 + \big( w\big)^2 \bigg)dz. 
\end{aligned}
$$
From the Gronwall inequality, \eqref{V_L1-norm} and \eqref{w_initial},
\begin{equation}
\label{est_w_dx}
|\partial_\xi w(\xi)|+|w(\xi)| \le c(|w(0)|+|\partial_\xi w(0)|)\le C\sigma^{\frac{1}{4}}(|v(T_0)|+|\partial_t v(T_0)|)
\end{equation}
for any $\xi>0$.

\medskip
Going back to $v$ and $\partial_t v$,
\begin{equation}\label{AAA}
|v(t)|+ |\partial_t v(t)|\le c\sigma^{\frac{1}{4}} Q(t)^{-\frac{1}{4}}(|v(T_0)|+|\partial_t v(T_0)|)
\end{equation}
for any $t>T_0$.

Now we proceed as in the previous subsections by selecting two solutions $v_3(t)$ and $\tilde{v}_3(t)$ of (\ref{eq-Jacobi-Toda-lin-EF}) in $I_3$, with 
$$
\begin{aligned}
v_3(T_0)=& 1 \quad \partial_{t}v_3(T_0)=&0\\
\tilde{v}_3(T_0)=& 0 \quad \partial_{t}\tilde{v}_3(T_0)=&1
\end{aligned}
$$
with Wronski determinant $W(t)=1$ in $I_3$. Clearly, $v_3(t)$ and $\tilde{v}_3(t)$ satisfy the estimate \eqref{AAA}.

\medskip
The same argument applied to $\tilde{v}$ yields for $t>T_0$ that
\begin{equation}\label{BBB}
\begin{aligned}
|\tilde{v}(t)|+ |\partial_t \tilde{v}(t)| &\le c\sigma^{\frac{1}{4}}(|\tilde{v}(T_0)|+|\partial_t \tilde{v}(T_0)|) Q(t)^{-\frac{1}{4}}
\end{aligned}
\end{equation}
Next, we solve equation (\ref{eq-Jacobi-Toda-lin-EF}) in $I_3$ setting the variations of parameters formula
$$
\begin{aligned}
u_3(t)=& u_2(T_0)v_3(t) + \partial_t u_2(T_0)\tilde{v}_3(t)\\
& + v_3(t)\int_{T_{0}}^t \tilde{v}_3(\tau)\tilde{f}(\tau)d\tau - \tilde{v}_3(t)\int_{T_{0}}^t {v}_3(\tau)\tilde{f}(\tau)d\tau 
\end{aligned}$$
for $t\in I_3$. 

Now it is possible to conclude the proof of Proposition \ref{propinverselinearJT} exactly as in Subsection $3.8$ in \cite{ABPRS}.

\subsection{The proof of Theorem \ref{main-th-JT}: a fixed point argument}\label{subs-fixed-point-JT}
In this section we find the solution $\h$ to system (\ref{JT-system}) by a fixed point argument. We look for a solution of the form $\h:={\tt v}^l+\q$, where $l>0$ is an integer, to be fixed later, ${\tt v}^l$ is the approximate solution constructed in Lemma \ref{lemmaapproximateslnJacToda} and $\q$ is a small correction.

With this assumption, a decoupling argument shows that system (\ref{JT-system}) is equivalent to system (\ref{syst-tilde-q}), as we proven in Subsection \ref{subs-decoupling}. Equivalently, we have to solve
\begin{equation}
\tilde{\q}_j=L_{\delta,j}^{-1}\mathcal{G}_{\delta,j}(\tilde{\q}),
\end{equation}
with respect to $\tilde{\q}$, where $L_{\delta,j}^{-1}$ is the inverse of the operator
$$L_{\delta,j}:=\Delta_\Sigma+|A_\Sigma|^2(1+\sqrt{2}(\mu_j^0+{\tt r}_j)\w)$$
constructed in Proposition \ref{propinverselinearJT} and $\tilde{\q}$ is defined in subsection \ref{subs-decoupling}.

Let $\sigma_0>0$ be as in Proposition \ref{propinverselinearJT} and let $\sigma:= |\log\delta|$ with $\sigma>\sigma_0$.

\medskip
By Lemma \ref{lemmaapproximateslnJacToda}, we have
\begin{equation}
\|E_{\delta,j}({\tt v}^l)\|_{\mathcal{D}^{2,\beta}_{2,\frac{l-1}{2}}(\Sigma)}\le C_l\sigma^{-\frac{l-1}{2}},
\end{equation}
thus, it is possible to find the solution $\tilde{\q}=(\tilde{\q}_1,\dots,\tilde{q}_{k-1})$ applying the contraction mapping Theorem in the ball
$$B:=\{\tilde{\q}\in(\mathcal{D}^{2,\beta}_{0,\frac{l-2}{2}}(\Sigma))^{k-1}:\, \|\tilde{\q}_j\|_{\mathcal{D}^{2,\beta}_{0,\frac{l-2}{2}}(\Sigma)}<\Lambda \sigma^{\frac{5-2l}{4}},\, 1\le j\le k-1\}$$
provided $\Lambda>0$ is large enough and $l$ is so large that $L^{-1}_j\mathcal{G}_{\delta,j}(\tilde{\q})\in B$. It is possible to prove that this condition is satisfied provided $l>7$. We refer to Subsections $3.8$ and $3.9$ of \cite{AKR1} for further details.

\section{The proof of Theorem \ref{main-th-AC}: construction and energy estimate}

In this section we fix $m,n\ge 2$ with $m+n\ge 8$ we prove Theorem \ref{main-th-AC}, finding a solution $u_{\eps}$ to  \eqref{eq_AC} whose zero level set is the disjoint union of $k$ normal graphs over a dilated version of $\Sigma:=\Sigma^-_{m,n}$.

Property ${\rm iv.}$ in \ref{main-th-AC} is proven in Section \ref{sec-Morse}.

\subsection{Fermi coordinates}

Let $\eps>0$ be a fixed scaling parameter and consider the dilated surface $\Sigma_\eps:=\eps^{-1}\Sigma$. Observe first that for any $\pe\in\Sigma_\eps$, there exists a unique 
$$
(\s,\x,\y)=(\s(\pe),\x(\pe),\y(\pe))\in \R \times S^{m-1}\times S^{n-1}
$$ 
such that
$$
\pe=\eps^{-1}\left(a(\eps\s)\x,b(\eps\s)\y\right).
$$
Denoting the function $\s$ by $s$ in the case $\eps=1$, we get
\begin{equation}\notag
s(\eps\pe)=\eps\s(\pe), \quad \forall\,\pe\in\Sigma_\eps,\,\eps>0.
\end{equation}
In the sequel, $\eps>0$ will be taken to be small.

Then the function
$$
X_\eps({\pe},{\tt z}):=\pe+{\tt z} \,\nu_\Sigma(\eps\pe)
$$
is a diffeomorphism onto the dilated neighbourhood
$$\mathcal{N}_\eps:=\left\{\pe+{\tt z}\,\nu_\Sigma(\eps\pe)\,:\,\pe\in\Sigma_\eps, \quad 
| {\tt z}|<\frac{\delta_0}{\eps}+\eta_0|\s(\pe)|\right\},
$$
provided $\delta_0>0$ is small enough and $0<\eta_0 < \frac{1}{2\sqrt{N-1}}\min (\sqrt{m-1},\sqrt{n-1})$. The new coordinates defined by this diffeomorphism are known as the \textit{Fermi coordinates} of $\Sigma_\eps$.

Fix $\alpha\in\left(0,\frac{1}{9}\right)$ and consider $k$ $O(m)\times O(n)$-invariant functions ${\tt h}_1,\dots, {\tt h}_k : \Sigma \to \R$ with ${\tt h}_l(\pe)=h_l(s(\pe))$ for $\pe \in \Sigma$, $1\le l\le k$ and such that
\begin{equation}
\begin{small}
\label{cond_growth_hl}
\frac{1}{\sqrt{2}}\left(l-\frac{3}{2}-\alpha\right)\left(\log(s^2+2)+2|\log\eps|\right)<h_l(s)
<\frac{1}{\sqrt{2}}\left(l-\frac{3}{2}+\alpha\right)\left(\log(s^2+2)+2|\log\eps|\right)
\end{small}
\end{equation}
for $s\in \R$. Assume also that ${\tt h}_l(\pe)=h_{l}(s)$ is even in the variable $s\in \R$ and $C^2(\Sigma)$. 

For $l=1,\dots,k$, the mapping
$$X_{\eps,\h_l}(\pe,t):=\pe+(t+{\tt h}_l(\eps\pe))\nu_\Sigma(\eps\pe)
$$
defines a diffeomorphism onto the tubular neighbourhood 
\begin{equation}
\label{metric_FC}
\begin{aligned}
\mathcal{N}_{l,\eps}
&=\left\{X_{\eps,\h_l}(\pe,t):|t|<\frac{1}{4\sqrt{2}}\left(\log(s(\eps\pe)^2+2)+2|\log\eps|\right)\right\}.\\
&:=\left\{X_\eps(\pe,{\tt z}):|{\tt z}-{\tt h}_l(\eps\pe)|<\frac{1}{4\sqrt{2}}\left(\log(s(\eps\pe)^2+2)+2|\log\eps|\right)\right\}.
\end{aligned}
\end{equation} 

Introducing the change of variables
\begin{equation}
\label{eps_shift_coord}
z=\eps(t+{\tt h}_l(\eps\pe))
\quad \hbox{for} \quad (\pe,t)\in X_{\eps,{\tt h}_l}^{-1}(\mathcal{N}_{l,\eps}),
\end{equation}
the euclidean Laplacian can be computed in the set $\mathcal{N}_{l,\eps}$ in the coordinates $X_{\eps,{\tt h}_l}(\pe,t)$.

Given a metric $g_{ij}$ on $\Sigma$, with inverse $g^{ij}$, through the Fermi coordinates we have a metric $G_{IJ}$ on $\mathcal{N}:=\mathcal{N}_1$ given by
$$
G_{ij}=g_{ij}-2A_{ij}{\tt z}+{\tt z}^2\partial_i\nu_\Sigma\cdotp\partial_j\nu_\Sigma \quad \hbox{for} \quad i,j =1, \ldots,N,$$
$$ 
G_{i{\tt z}}=G_{{\tt z}i}=0 \quad \hbox{for} \quad i=1,\ldots,N, \quad G_{{\tt z}{\tt z}}=1,
$$
with inverse $G^{IJ}$. We denote the mean curvature and the Laplace-Beltrami operator of the normally translated hypersurface
$$
\Sigma_{\tt z}:=\{\pe+{\tt z}\nu_\Sigma(\pe):\, \pe\in\Sigma\}
$$
by $H_{\Sigma_{\tt z}}$ and $\Delta_{\Sigma_{\tt z}}$ respectively. Then we introduce the notations
$${\tt Q}(\pe,z):=H_{\Sigma_z}-|A_\Sigma|^2 \z$$
and
\begin{equation}
\begin{aligned}
a^{ij}(\pe,\z)&:=G^{ij}-g^{ij}\\
b^j(\pe,{\tt z})&:=\partial_i G^{ij}(\pe,{\tt z})+\partial_i(\log\sqrt{\det G(\pe,{\tt z})})G^{ij}(\pe,{\tt z})-\partial_i g^{ij}(\pe)-\partial_i(\log\sqrt{\det g})g^{ij}(\pe).
\end{aligned}
\end{equation}
so that
$$a^{ij}\partial_{ij}+b^i\partial_i=\Delta_{\Sigma_z}-\Delta_\Sigma.$$
\begin{lemma}
\label{lemma_coord_Fermi}
In the neighbourhood $\mathcal{N}_{l,\eps}$, the Laplacian in the $(\pe,t)$ coordinates is given by
\begin{equation}
\label{Laplacian_Fermi}
\begin{aligned}
\Delta&=\Delta_{\Sigma_\eps}+\partial^2_t-\eps^2\big(\Delta_\Sigma {\tt h}_l+|A_\Sigma|^2{\tt h}_l\big)\partial_t- \eps^2 t|A_{\Sigma}|^2 \partial_{t} - 2 \eps\nabla_\Sigma {\tt h}_l\cdotp\partial_t\nabla_\Sigma+\eps^2|\nabla_\Sigma {\tt h}_l|^2\partial^2_t
\\
&-\eps{\tt Q}\partial_t +(a^{ij}\partial^2_{ij}+\eps b^j\partial_j)-\eps^2(a^{ij} \partial_{ij}{\tt h}_l+b^j \partial_j{\tt h}_l)\partial_t-2\eps a^{ij} \partial_i{\tt h}_l\partial_{tj}+\eps^2a^{ij}\partial_i{\tt h}_l\partial_j{\tt h}_l \partial^2_t,
\end{aligned}
\end{equation}
where $|A_\Sigma|^2$, ${\tt h}_l$ and its derivatives are evaluated at $\eps\pe$, while ${\tt Q},\, a^{ij}$ and $b^j$ are evaluated at $(\eps\pe,\eps(t+{\tt h}_l(\eps\pe)))$.
\end{lemma}
Since we are interested in $O(m)\times O(n)$ invariant functions, we write the Laplacan in the $(\s,t)$ variables, that is
\begin{equation}
\label{Laplacian_Fermi_st}
\begin{aligned}
\Delta= & \Delta_{\Sigma_{\eps}}+\partial^2_t-\eps^2(\Delta_{\Sigma}{\tt h}_l +|A_\Sigma|^2{\tt h}_l)\partial_t - \eps^2t\beta \partial_t-\eps^2(h''_l+\alpha(\eps \s)h'_l)\partial_t-2\eps h'_l\partial_{t\s}\\
&+\eps^2(h'_l)^2\partial^2_t-\eps\Q\partial_t+\ta\partial^2_\s+\eps\bi\partial_\s-\eps^2(\ta h''_l+\bi h'_l)\partial_t-2\eps\ta h'_l\partial_{t\s}+\eps^2\ta(h'_l)^2 \partial^2_t,
\end{aligned}
\end{equation}
where
$$
\Delta_{\Sigma_{\eps}}{\tt h}_l +|A_\Sigma|^2{\tt h}_l=h''_l+\eps\alpha(\eps\s)h'_l+\beta(\eps\s)h_l
$$
$\beta$, $h_l$ and its derivatives are evaluated at $\eps\s$, while $\Q,\, a^{ij}$ and $b^j$ are evaluated at $(\eps\s,\eps(t+h_l(\eps\s)))$ and fulfil
\begin{equation}
\label{est_remainder_laplacian}
\begin{aligned}
|\Q(\eps\s,\eps(t+h_l(\eps\s)))|&\le C\frac{\eps^2(t+h_l(\eps\s))^2}{(1+|\eps \s|)^3}\\
|\ta(\eps\s,\eps(t+h_l(\eps\s)))|&\le C\frac{\eps(t+h_l(\eps\s))}{1+|\eps\s|}
\\
|\bi(\eps\s,\eps(t+h_l(\eps\s)))|&\le C\frac{\eps(t+h_l(\eps\s))}{(1+|\eps\s|)^2}.
\end{aligned}
\end{equation}
For the proof of Lemma \ref{lemma_coord_Fermi} and the estimates (\ref{est_remainder_laplacian}) see Section 4.1 of \cite{AKR1}.

\subsection{Approximate solution to the Allen-Cahn equation \ref{eq_AC}}

In this subsection we describe the approximate solution to the Allen-Cahn equation (\ref{eq_AC}). In order to do so, first we fix $k$ $O(m)\times O(n)$ invariant $C^2$ functions $\h_1,\dots,\h_k:\Sigma\to\R$ such that the corresponding functions $h_j(s(\pe)):=\h_j(\pe)$ fulfil (\ref{cond_growth_hl}) and
\begin{equation}
\label{cond_growth_der_hl}
|h'_j(s)|\le\frac{c}{|s|+1},\quad |h''_j(s)|\le\frac{c}{(|s|+1)^2}, \quad 1\le j\le k,\,s\in\R.
\end{equation}
Then, using the Fermi coordinates $(\pe,\z)$ of $\Sigma_\eps$ we set
\begin{equation}
U_0(\pe,z)=\sum_{j=1}^k w_j(z-\h_j(\eps\pe))+\frac{(-1)^{k-1}-1}{2}, \qquad w_j(t):=(-1)^{j-1}v_\star(t),
\end{equation}
being $v_\star(t):=\tanh(t/\sqrt{2})$. Set also
$$
S(u)=\Delta u+F(u), \qquad F(u)=u(1-u^2)
$$
and let us now compute the error $S(U_0)$ near the normal graphs of the functions $\pe\mapsto{\tt h}_l(\eps\pe)$ over $\Sigma_\eps$.

\begin{lemma}\label{lemmaerrorSUzero}
For $1\le l\le k$ and for any $(\pe,t)\in X^{-1}_{\eps, {\tt h}_l}(\mathcal{N}_{l,\eps})$,
\begin{equation}\label{errorSUzero}
\begin{aligned}
(-1)^{l-1}S(U_0)=&-\eps^2(\Delta_\Sigma \h_l+|A_\Sigma|^2 \h_l)v'_\star-\eps^2|A_\Sigma|^2tv'_\star+\eps^2|\nabla_\Sigma \h_l|^2v''_\star\\
&+6(1-v_\star^2)(e^{-\sqrt{2}t}e^{-\sqrt{2}(\h_l-\h_{l-1})}-e^{\sqrt{2}t}e^{-\sqrt{2}(\h_{l+1}-\h_l)})
+R_\eps(\h_1,\dots,\h_k),
\end{aligned}
\end{equation}
where $R_\eps(\h_1,\dots,\h_k)$ is such that
\begin{equation}\label{estimateReps}
|R_\eps(\h_1,\dots,\h_k)|\le C\eps^{2+\gamma}(s(\eps\pe)^2+2)^{-\frac{2+\gamma}{2}}e^{-\rho|t|} \quad \hbox{in} \quad X_{\eps,\h_l}^{-1}(\mathcal{N}_{l,\eps})
\end{equation}
for some $\gamma\in(0,\frac{1}{2})$ and $\rho\in(0,\sqrt{2})$.
\end{lemma}
Now we improve the approximation, that is we look for a better approximate solution in such a way that the term of order $\eps^2$ in the error is cancelled. In order to do so we write
\begin{equation}\label{AClinearOrthCond1}
6(1-v_\star^2)e^{-\sqrt{2}t}=a_{\star} v'_\star+g_{0}(t) \quad  \hbox{with} \quad \int_\R g_{0}(t)v'_\star(t)dt=0
\end{equation}
and, also using that
\begin{equation}\label{AClinearOrthCond2}
\int_{\R} v''_{\star}(t)v'_{\star}(t)dt =\int_\R t(v'_\star(t))^2 dt=0.
\end{equation}
we solve the ODEs
\begin{equation}
\psi''_0+(1-3v_\star^2)\psi_0=g_0, \quad  \psi''_1+(1-3v_\star^2)\psi_1=-v''_\star, \quad 
\psi''_2+(1-3v_\star^2)\psi_2=tv'_\star \quad \hbox{in} \quad \R.
\end{equation}
under the orthogonality condition
$$\int_\R \psi_i v'_\star =0\qquad 0\le i\le 2.$$
Then we define $\eta_j: \Sigma_{\eps}\times \R \to \R$ by the formula 
\begin{equation}
\begin{aligned}
(-1)^{j-1}\eta_j(\pe,\z):=&-e^{-\sqrt{2}(\h_j(\eps\pe)-\h_{j-1}(\eps\pe))}\psi_0(\h_j(\eps\pe)-\z)
+e^{-\sqrt{2}(\h_{j+1}(\eps\pe)-\h_j(\eps\pe))}\psi_0(\z-\h_j(\eps\pe))\\
&\hspace{3cm}+\eps^2|\nabla_\Sigma \h_j(\eps\pe)|^2\psi_1(\z-\h_j(\eps\pe))+\eps^2|A_\Sigma(\eps\pe)|^2\psi_2(\z-\h_j(\eps\pe))
\end{aligned}
\end{equation}
and we consider the approximation
\begin{equation}\notag
U_1(\pe,\z):=U_0(\pe,\z)+\eta(\pe,\z), \qquad \eta(\pe,\z):=\sum_{j=1}^k\eta_{j}(\pe,\z). 
\end{equation}
\begin{lemma}
\label{lemma_new_error}
Assume the hypothesis in Lemma \eqref{errorSUzero}. The error $S(U_1)$ in $X^{-1}_{\eps,\h_l}(\mathcal{N}_{\eps,l})$ is given by
\begin{equation}\notag
\begin{aligned}
(-1)^{l-1}S(U_1)=-\eps^2(\Delta_\Sigma \h_l+|A_\Sigma|^2 \h_l)v'_\star+a_\star(e^{-\sqrt{2}(\h_l-\h_{l-1})}-e^{-\sqrt{2}(\h_{l+1}-\h_l)})v'_\star+R_{6,\eps}(\h_1,\dots,\h_k),
\end{aligned}
\end{equation}
with 
\begin{equation}
|R_{6,\eps}(\h_1,\dots,\h_k)|\le C \eps^{2+\gamma}(s(\eps\pe)^2+2)^{-\frac{2+\gamma}{2}}e^{-\rho|t|}.
\end{equation}
\end{lemma}

\subsection{A gluing procedure}\label{section-gluing}

In this subsection we define a global approximation and we reduce the Allen-Cahn equation (\ref{eq_AC}) to a system of $k+1$ equation, through a gluing argument. Throughout the section, we will use the notations of Subsection $5.1$ of \cite{AKR1}. 

For the reader's convenience, we recall that we took a cutoff function $\chi\in C^\infty(\R)$ such that $0\le\chi\le 1$ and
\begin{equation}\notag
\chi(t)=\begin{cases}
1 \qquad t\le 1\\
0 \qquad t\ge 2.
\end{cases}
\end{equation}
and, for $(\pe,{\tt z})\in\Sigma_\eps\times\R$, we set
$$\zeta(\pe,{\tt z}):=\chi\left(|{\tt z}|-\frac{4}{\sqrt{2}}(\log(s(\eps\pe)^2+2)+2|\log\eps|)+2\right).$$
Composing with the inverse of $X_\eps$, we have the cutoff function
\begin{equation}
\tilde{\zeta}(\xi):=
\left\{\begin{aligned}
\zeta\circ X_\eps^{-1}(\xi) , \qquad &\hbox{for} \quad \xi\in \mathcal{N}_\eps\\
0 \,\,\,\,\,\,, \qquad &\hbox{otherwise.}
\end{aligned}
\right.
\end{equation}
defined in natural Euclidean coordinates.
Using this function, we set
\begin{equation}\label{globalapproxsect5}
\tilde{U}_1(\xi):=
\left\{\begin{aligned}
U_1\circ X_\eps^{-1}(\xi) , \qquad &\hbox{for} \quad \xi\in \mathcal{N}_\eps\\
0 \,\,\,\,\,\,, \qquad &\hbox{otherwise.}
\end{aligned}
\right.
\end{equation}
Moreover, we introduce a function 
$$\mathbb{H}(\xi):=
\begin{cases}
-1 \qquad\text{in $E^-$}\\
(-1)^{k-1} \qquad\text{in $E^+$}
\end{cases}
$$
and we choose
$$w:=\tilde{\zeta}\tilde{U}_1+\mathbb{H}(1-\tilde{\zeta})$$
as a global approximation. The idea is that $w$ is constant far from the surface, and oscillates $k$ times in a neighbourhood of $\Sigma$ which is logarithmically opening as $\eps\to 0$ and $|\xi|\to\infty$.  

For $1\le l\le k$ and a function $u:\Sigma_\eps\times\R\to\R$, we define its \textit{pull back} by
\begin{equation}
u^{\natural}_l(\xi):=
\begin{cases}
u\circ X_{\eps,{\tt h}_l}^{-1}(\xi),\qquad &\text{for $\xi\in\mathcal{N}_{\eps,l}\subset\R^{N+1}$}\\
\hspace{1cm}0, \qquad &\text{otherwise}.
\end{cases}
\end{equation}

On the other hand, given any $O(m)\times O(n)$-invariant function $v:\R^{N+1}\to\R$, we define for $1\le l\le k$ and $(\pe,t)\in\Sigma_\eps\times\R$, the \textit{push forward}
\begin{equation}
v_{l}^{\sharp}(\pe,t):=\begin{cases}
v\circ X_{\eps,{\tt h}_l}(\pe,t)\qquad&\text{if $(\pe,t)\in X^{-1}_{\eps,{\tt h}_l}(\mathcal{N}_{\eps,l})$}\\
\hspace{1cm}0,\qquad &\text{otherwise.}
\end{cases}
\end{equation}

\medskip
For any integer $i\ge 1$ and $(\pe,t)\in\Sigma_\eps\times\R$, we set
\begin{equation}\label{def:chi_j}
\chi_i(\pe,t):=\chi\left(|t|-\frac{1}{4\sqrt{2}}(\log(s(\eps\pe)^2+2)+2|\log\eps|)+i\right)
\end{equation}

We look for a solution to (\ref{eq_AC}) of the form
$$u=w+\varphi,\qquad\varphi=\sum_{l=1}^k \chi^{\natural}_{3,j}\varphi_j+\psi,$$
so that, arguing exactly as in \cite{AKR1}, it is possible to see that the Allen-Cahn equation \ref{eq_AC} is equivalent to the system given by
\begin{equation}
\label{eq_nl_psi}
\begin{aligned}
&\Delta\psi+\left(2-(1-\sum_{j=1}^k \chi^{\natural}_{4,j})(F'(w)+2)\right)\psi+\left(1-\sum_{j=1}^k \chi^{\natural}_{4,j}\right)S(w)\\
&+\sum_{j=1}^k 2\nabla\chi^{\natural}_{3,j}\cdotp\nabla\varphi_j+\Delta\chi^{\natural}_{3,j}\varphi_j
+(1-\chi^{\natural}_{4,j})Q_w\left(\psi+\sum_{i=1}^k \chi^{\natural}_{3,i}\varphi_i\right)=0 \quad \hbox{in} \quad \R^{N+1}
\end{aligned}
\end{equation}
and
\begin{equation}
\label{eq_nl_phi_j}
\begin{aligned}
&\Delta_{\Sigma_\eps}\phi_j+\partial^2_t\phi_j+F'(v_\star)\phi_j+\chi_4 S(w^{{\sharp}}_j)+\chi_4
Q_{w^{\sharp}_{j}}(\phi_j+\psi^{\sharp}_{j})+\chi_4(F'(w^{\sharp}_{j})+2)\psi^{\sharp}_{j}\\
&\qquad+\chi_2(\Delta_{\mathcal{N}_{j,\eps}}-\partial^2_t-\Delta_{\Sigma_\eps})\phi_j+\left(F'(w^{\sharp}_{j})-F'(v_\star)\right)\chi_2\phi_j=0 \quad \hbox{in} \quad {\Sigma_{\eps} \times \R},\quad  1\le j\le k,
\end{aligned}
\end{equation}
whose unknowns are the functions $\phi_j:\Sigma_\eps\times\R\to\R$, $1\le j\le k$, which play the role of corrections close to the nodal set, the functions $\h_j:\Sigma\to\R$, used in the definition of the approximate solution $w$, and $\psi:\R^{N+1}\to\R$, which is the correction far from the nodal set. 

First we find a solution $\psi=\psi(\phi_1,\dots,\phi_k,\h_1,\dots,\h_k)$ to (\ref{eq_nl_psi}), for any given $\phi_1,\dots,\phi_k$ and $\h_1,\dots,\h_k$, then we plug it into system \ref{eq_nl_phi_j} and we solve it through an infinite dimensional Lyapunov-Schmidt reduction.

\subsection{The correction far from $\Sigma$}\label{sec-eq-far}

Arguing exactly as in Lemma $5.1$ of \cite{AKR1}, we can see that
\begin{equation}
\label{est-error-far}
\left|\left(1-\sum_{l=1}^2 \chi^{\natural}_{4,l}\right)S(w)\right|\le C\eps^{2+\bar{\gamma}}(|\eps\xi|^2+2)^{-\frac{2+\bar{\gamma}}{2}},\qquad\forall\,\xi\in\R^{N+1}.
\end{equation}
therefore we can use the same function spaces introduced in \cite{ABPRS}. For the reader's convenience, we recall their definition here.

For $\beta\in(0,1)$, $\mu>0$ and functions $g\in C^{0,\beta}_{loc}(\R^{N+1})$, we introduce the norm
\begin{equation}\notag
\|g\|_{{\infty, \mu}}:=\|(|\eps\xi|^2+2)^{\frac{\mu}{2}}g\|_{L^\infty(\R^{{N+1}})}.
\end{equation}
Moreover, we say that $g\in{\tt Y}^{{\natural,\beta}}_{{\mu}}$ if it is $O(m)\times O(n)$-invariant and the norm
\begin{equation}
\|g\|_{{\tt Y}^{{\natural,\beta}}_{{\mu}}}:=\sup_{\xi\in\R^{N+1}}(2+|\eps\xi|^2)^{\frac{2+\mu}{2}}\|g\|_{C^{0,\beta}(B_1(\xi))}
\end{equation}
is finite.

We also say that a function $\psi\in C^{2,\beta}_{loc}(\R^{N+1})$ is in ${\tt X}^{{\natural,\beta}}_{{\mu}}$ if it is $O(m)\times O(n)$-invariant and the norm
\begin{equation}
\label{norm_C2pmu}
\|\psi\|_{{\tt X}^{{\natural,\beta}}_{{\mu}}}:=\sup_{\xi\in\R^{N+1}}(2+|\eps\xi|^2)^{\frac{2+\mu}{2}}\|D^2\psi\|_{C^{0,\beta}(B_1(\xi))}
+\|\nabla\psi\|_{{\infty, 2+\mu}}+\|\psi\|_{{\infty, 2+\mu}}
\end{equation}
is finite.

\medskip
Given $\beta\in (0,1)$, $\rho\in(0,\sqrt{2})$, $\mu>0$ and a function $f\in C^{0,\beta}_{loc}(\Sigma_\eps\times\R)$, we define the norm
\begin{equation}
\|f\|_{{\infty, \mu,\rho}}:=\|(s(\eps\pe)^2+2)^{\frac{2+\mu}{2}}\cosh(t)^\rho f\|_{L^\infty(\Sigma_\eps\times\R)},
\end{equation} 
Furthermore, we say that $f\in Y^{{\sharp,\beta}}_{{\mu, \rho}}$ if it is $O(m)\times O(n)$-invariant and
\begin{equation}
\|f\|_{Y^{{\sharp,\beta}}_{{\mu, \rho}}}:=\sup_{\pe\in \Sigma_\eps,\,t\in\R}(s(\eps\pe)^2+2)^{\frac{2+\mu}{2}}\cosh(t)^\rho \|f\|_{C^{0,\beta}(I_{\pe,t})},\qquad I_{\pe,t}:=B_1(\pe)\times(t,t+1) 
\end{equation}
is finite. Moreover, for $O(m)\times O(n)$-invariant functions $\phi\in C^{2,\beta}_{loc}(\Sigma_\eps\times\R)$, we say that $\phi\in X^{{\sharp,\beta}}_{{\mu, \rho}}$ if
\begin{equation}
\|\phi\|_{X^{{\sharp,\beta}}_{{\mu, \rho}}}:=\sup_{\pe\in\Sigma_\eps,\,t\in\R}(s(\eps\pe)^2+2)^{\frac{2+\mu}{2}}\cosh(t)^\rho \|D^2\phi\|_{C^{0,\beta}(I_{\pe,t})}+\|\nabla\phi\|_{{\infty, 2+\mu,\rho}}
+\|\phi\|_{{\infty, 2+\mu,\rho}}
\end{equation}
is finite.

We make the assumption that $\h:=(\h_1,\dots,\h_k):\Sigma\to\R$ is of the form
\begin{equation}
\label{choice-h}
\h:={\tt v}+\q,\qquad \q:=B^{-1}\hat{\q}
\end{equation}
where ${\tt v}:=({\tt v}_1,\dots,{\tt v}_k)$ is the solution to the Jacobi-Toda system found in Theorem \ref{main-th-JT}, the $k\times k$ matrix $B$ is defined in \ref{def-B} and we take $\hat{\q}:=(\hat{\q}_1,\dots,\hat{\q}_{k})$ such that $\hat{\q}_j\in\mathcal{D}^{2,\beta}_{\mu,\frac{1}{2}}(\Sigma)$ for $1\le j\le k-1$ and $\hat{\q}_k\in\mathcal{C}^{2,\beta}_{\infty,\mu}(\Sigma)$.

\begin{proposition}
Let $\beta\in(0,\frac{1}{2})$, $\rho>0$ and let $\Lambda_0,\,\Lambda_1>0$ be fixed constants. Let also $\h$ be of the form (\ref{choice-h}), with $\hat{\q}_k\in\mathcal{C}^{2,\beta}_{\infty,\mu}(\Sigma),\,\hat{\q}_j\in\mathcal{D}^{2,\beta}_{\mu,\frac{1}{2}}(\Sigma),\,1\le j\le k-1$ such that
$$\|\hat{\q}_k\|_{\mathcal{C}^{2,\beta}_{\infty,\mu}(\Sigma)}<\Lambda_0\eps^{\mu},\qquad
\|\hat{\q}_j\|_{\mathcal{D}^{2,\beta}_{\mu,\frac{1}{2}}(\Sigma)}<\Lambda_0\eps^{\mu},\, 1\le j\le k-1.$$
Let $\phi:=(\phi_1,\dots,\phi_k)\in X^{{\sharp,\beta}}_{{\mu, \rho}}$ be such that 
$$\|\phi_j\|_{X^{{\sharp,\beta}}_{{\mu, \rho}}}<\Lambda_1\eps^{2+\mu},\qquad 1\le j\le k.$$ 
Then there exists a unique solution $\psi:=\psi(\phi,\h)\in {\tt X}^{{\sharp,\beta}}_{{\mu, \rho}}$ to equation (\ref{eq_nl_psi}) satisfying
\begin{equation}\notag
\begin{aligned}
\|\psi(\phi,\h)\|_{{\tt X}^{{\sharp,\beta}}_{{\mu, \rho}}}  &\le\Lambda_2\eps^{2+\mu},\\
\|\psi(\phi^1,\h)-\psi(\phi^2,\h)\|_{{\tt X}^{{\sharp,\beta}}_{{\mu, \rho}}}&\le c\eps^{2+\mu}\sum_{j=1}^k\|{\phi^1_j-\phi^2_j}\|_{X^{\sharp,\beta}_{\mu,\frac{1}{2}}(\Sigma)}\\
\|\psi(\phi,\h^1)-\psi(\phi,\h^2)\|_{{\tt X}^{{\sharp,\beta}}_{{\mu, \rho}}}&
\le c\eps^{2+\mu}\left(\|\hat{\q}^1_k-\hat{\q}^2_k\|_{\mathcal{C}^{2,\beta}_{\infty,\mu}(\Sigma)}
+\sum_{j=1}^{k-1}\|\q^1_j-\q^2_j\|_{\mathcal{D}^{2,\beta}_{\mu,\frac{1}{2}}(\Sigma)}\right),
\end{aligned}
\end{equation}
for some $c,\,\Lambda_2>0$. 
\label{prop_psi}
\end{proposition}
For a proof, see Proposition $5.1$ of \cite{AKR1}.

\subsection{The Lyapunov-Schmidt reduction}\label{sec-Lyapunov-Schmidt}
In this section we solve system (\ref{eq_nl_phi_j}). We fix $\h$ as in (\ref{choice-h}) and we set
\begin{equation}
\begin{aligned}
Q_w(\varphi)&:=F(w+\varphi)-F(w)-F'(w)\varphi\\
{\tt N}_{l}^{\sharp}(\psi,\phi_1,\dots,\phi_k,\h_1,\dots,\h_k)&:=
\chi_4Q_{w_{l}^{{\sharp}}}(\phi_l+\psi_{l}^{{\sharp}})+\chi_4(F'(w_{l}^{{\sharp}})+2)\psi_{l}^{{\sharp}}+\chi_2(\Delta_{\mathcal{N}_{l,\eps}}-\partial^2_t-\Delta_{\Sigma_\eps})\phi_l\\
&\qquad+(F'(w_{l}^{{\sharp}})-F'(v_\star))\chi_2\phi_l,\\
P_{l}^{\sharp}(\psi,\phi_1,\dots,\phi_2,\h_1,\dots,\h_2)(\pe)&:=\int_\R \left(\chi_4 S(w_{l}^{{\sharp}})+{\tt N}_{l}^{\sharp}(\psi,\phi_1,\dots,\phi_k,\h_1,\dots,\h_k)\right)v'_\star(t) dt \quad\hbox{for} \quad\pe\in\Sigma_\eps.
\end{aligned}
\end{equation}
first we find a solution $\phi:=(\phi_1,\dots,\phi_k):\Sigma_\eps\times\R\to\R$ to the system
\begin{equation}
\label{eq_nl_phi_j_pr}
\begin{aligned}
\Delta_{\Sigma_\eps}\phi_l+\partial^2_t\phi_l+F'(v_\star)\phi_l=-\chi_4 S(w_{l}^{{\sharp}})-&{\tt N}_{l}^{\sharp}(\psi,\phi_1,\dots,\phi_k,\h_1,\dots,\h_k)+P_{l}^{\sharp}(\psi,\phi_1,\dots,\phi_k,\h_1,\dots,\h_k)(\pe)v'_\star, \\
\int_\R \phi_l(\pe,t)v'_\star(t)dt=&0,\qquad\forall\, \pe\in\Sigma_\eps, \quad  1\le l \le k
\end{aligned}
\end{equation}
where $\psi=\psi(\phi,\h)$ is the solution found above, then we look for $\h:=(\h_1,\dots,\h_k)$ such that
\begin{equation}
\label{bifo_eq}
P_{l}^{\sharp}(\psi,\phi_1,\dots,\phi_k,\h_1,\dots,\h_k)=0, \qquad 1\le l\le k,
\end{equation}
System (\ref{eq_nl_phi_j_pr}) is known as the \textit{auxiliary equation}, while system (\ref{bifo_eq}) is known as the \textit{bifurcation equation}.

\begin{proposition}
Let $\beta\in(0,\frac{1}{2})$, $\rho>0$ be given and let $\h_1,\dots,\h_k$ be 
of the form (\ref{choice-h}), with $\hat{\q}_k\in\mathcal{C}^{2,\beta}_{\infty,\mu}(\Sigma),\,\hat{\q}_j\in\mathcal{D}^{2,\beta}_{\mu,\frac{1}{2}}(\Sigma),\, 1\le j\le k-1$ such that
$$\|\hat{\q}_k\|_{\mathcal{C}^{2,\beta}_{\infty,\mu}(\Sigma)}<\Lambda_0\eps^{\mu},\qquad
\|\hat{\q}_j\|_{\mathcal{D}^{2,\beta}_{\mu,\frac{1}{2}}(\Sigma)}<\Lambda_0\eps^{\mu},\,1\le j\le k-1$$
with $\Lambda_0>0$ fixed. Then there exists a unique solution $\phi=\phi(\h)\in (X_{\bot,\mu,\rho}^{\sharp,\beta})^k$ satisfying
\begin{equation}
\begin{aligned}
\|\phi_j(\h)\|_{X_{\mu,\rho}^{\sharp,\beta}}&\le \Lambda_1 \eps^{2+\mu},\,\qquad\forall\, 1\le j\le k,\\
\|\phi_j(\h^1)-\phi_j(\h^2)\|_{X_{\mu,\rho}^{\sharp,\beta}}&\le c\eps^{2+\mu}(\|\hat{\q}^1_k-\hat{\q}^2_k\|_{\mathcal{C}^{2,\beta}_{\infty,\mu}(\Sigma)}
+\sum_{i=1}^{k-1}\|\hat{\q}^1_i-\hat{\q}^2_i\|_{\mathcal{D}^{2,\beta}_{\mu,\frac{1}{2}}(\Sigma)}), \qquad 1\le j\le k-1,
\end{aligned}
\end{equation}
for some $c,\,\Lambda_1>0$.
\label{prop_phi_l}
\end{proposition}
For a proof, see Proposition $5.3$ of \cite{AKR1}.

Now we consider the bifurcation equation (\ref{bifo_eq}). Computing $P_{l}^{\sharp}(\psi,\phi_1,\dots,\phi_k,\h_1,\dots,\h_k)$ explicitly, through the expansion of the Laplacian given by (\ref{Laplacian_Fermi_st}), it is possible to see that the bifurcation equation \ref{bifo_eq} is equivalent to the Jacobi-Toda type system ($\delta=\eps^2$ in system \eqref{JT-system})
\begin{equation}
\label{eq-bifo-proj}
\eps^2(\Delta_\Sigma \h_j+|A_\Sigma|^2 \h_j)-a_\star(e^{-\sqrt{2}(\h_j-\h_{j-1})}-e^{-\sqrt{2}(\h_{j+1}-\h_j)})=\eps^2 f_j(\pe,\h_1,\dots,h_k)\qquad\text{in $\Sigma$, $1\le j\le k$}
\end{equation}
with $f_j$ fulfilling
\begin{equation}
\label{dec_f_l}
|f_j(\pe,\h_1,\dots,\h_k)|\le c\eps^{\mu}(s(\pe)^2+2)^{2+\mu} \quad \hbox{for} \quad \pe\in\Sigma,\quad\,1\le j\le k,
\end{equation}
where $\,\mu:=\min\{\gamma,\bar{\gamma}\}>0$ with $\gamma$ and $\bar{\gamma}$ begin the constants in Lemma \ref{lemmaerrorSUzero} and relation (\ref{est-error-far}) and 
$$
a_{\star}= \|v'_{\star}\|^{-2}_{L^2(\R)}\int_{\R}6(1-v_\star^2)e^{-\sqrt{2}t}v'_{\star}(t)dt>0
$$
is the constant in \eqref{AClinearOrthCond1}. 

Arguing as in Section \ref{subs-decoupling}, setting $\q^\sharp:=(\q_1,\dots,\q_{k-1})$, it is possible to rewrite (\ref{eq-bifo-proj}) as
\begin{equation}
\label{eq-bifo-sharp}
\begin{aligned}
J_\Sigma {\q}^\sharp+a_\star \sqrt{2}A \q^\sharp&=
\hat{\hat{\mathcal{F}}}_{\eps^2,j}(\cdotp,B^{-1}(\q^\sharp,\hat{\q}_k))+\hat{f}_j(\pe,{\tt v}+B^{-1}(\q^\sharp,\hat{\q}_k))\qquad 1\le j\le k-1\\
J_\Sigma \hat{\q}_k&=\hat{f}_k(\pe,{\tt v}+B^{-1}(\q^\sharp,\hat{\q}_k))
\end{aligned}
\end{equation}
where $\hat{f}(\cdotp,\q):=Bf(\cdotp,\q)$, the $(k-1)\times(k-1)$ matrix $A$ is defined in (\ref{def-A}) and $\mathcal{F}_{\eps^2}$ is defined in (\ref{def-remainder-JT}) and $\hat{\mathcal{F}}_{\eps^2}=B\mathcal{F}_{\eps^2}$. Arguing as in Section \ref{subs-decoupling}, we observe that $A$ is similar to a diagonal matrix $M$ with positive eigenvalues, namely $A=VMV^{-1}$, so that (\ref{eq-bifo-sharp}) can be rewritten in terms of $(\tilde{\q},\hat{\q}_k):=(V\q^\sharp,\hat{\q}_k)$ as
\begin{equation}
\label{eq-bifo-to-solve}
\begin{aligned}
\Delta_\Sigma \tilde{\q}_j+|A_\Sigma|^2(1+\sqrt{2}(\mu_j^0+{\tt r}_j))\tilde{\q}_j&=\mathcal{G}_{\eps^2,j}(\tilde{\q})+
V^{-1}\hat{f}_j(\pe,{\tt v}+B^{-1}(V^{-1}\tilde{\q},\hat{\q}_k))\qquad 1\le j\le k-1\\
J_\Sigma \hat{\q}_k&=\hat{f}_k(\pe,{\tt v}+B^{-1}(V^{-1}\tilde{\q},\hat{\q}_k)),
\end{aligned}
\end{equation}
where $\mathcal{G}_{\eps^2}$ is defined in (\ref{def-G}). Arguing as in Section $5.5$ of \cite{AKR1}, it is possible to solve system (\ref{eq-bifo-to-solve}) in a ball
$$B_{\Lambda_0}:=\{(\tilde{\q},\hat{\q}_k)\in (\mathcal{D}^{2,\beta}_{\mu,\frac{1}{2}}(\Sigma))^{k-1}\times \mathcal{C}^{2,\beta}_{\infty,\mu}(\Sigma):\,
\|\tilde{\q}_j\|_{\mathcal{D}^{2,\beta}_{\mu,\frac{1}{2}}(\Sigma)}<\Lambda_0\eps^\mu,\,1\le j\le k-1,\,\|\tilde{\q}_k\|_{\mathcal{C}^{2,\beta}_{\infty,\mu}(\Sigma)}<\Lambda_0\eps^\mu\},$$
provided $\Lambda_0>0$ is large enough. This concludes the construction of a family $u_\eps$ of solutions to (\ref{eq_AC}) satisfying Properties (\ref{geom-inv}) and (\ref{k-end}) of Theorem (\ref{main-th-AC}). Property (\ref{energy-estimate}) of Theorem \ref{main-th-AC} follows arguing exactly as in Section $6$ of \cite{AKR1}.

\section{The Morse index}\label{sec-Morse}

Recall that $\sigma:=|\ln(\eps^2)|$, for $\eps>0$ small. In order to prove that the Morse index is infinite for any $\eps\in(0,\eps_0)$, it is enough to prove the existence of an infinite dimensional space $X$ of compactly supported functions such that the quadratic form 
\begin{equation}
\mathcal{Q}_\sigma(\phi)=\int_\Sigma |\nabla_\Sigma\phi|^2-|A_\Sigma|^2(1+\sigma)\phi^2
\end{equation}
satisfies
\begin{equation}
\label{est-Qsigma}
\mathcal{Q}_\sigma(\phi)<-\int_\Sigma |A_\Sigma|^2\phi^2\qquad\forall\,\phi\in X.
\end{equation}
In order to do so, we first consider the eigenvalue problem for the cone, which reads
\begin{equation}
\label{eigenv-cone}
-\phi_{rr}-\frac{N-1}{r}\phi_r-\frac{N-1}{r^2}\phi=\lambda\frac{N-1}{r^2}\phi
\end{equation}
Introducing the change of variables $r=e^t$ and writing $\phi(r)=v(t)$, equation (\ref{eigenv-cone}) reads
\begin{equation}
\label{eigenv-cone-v}
v_{tt}+(N-2)v_t+(N-1)(1+\lambda)v=0.
\end{equation}
We look for solutions of the form $v(t)=e^{\gamma t}$. A direct calculation shows that $v(t)$ is a solution if and only if
$$\gamma=\gamma_\pm(\lambda):=-\frac{N-2}{2}\pm\sqrt{\left(\frac{N-2}{2}\right)^2-(N-1)(1+\lambda)}=-\frac{N-2}{2}\pm i\sqrt{\alpha(\lambda)}$$
where
$$\alpha(\lambda):=-\left(\frac{N-2}{2}\right)^2+(N-1)(1+\lambda)>0$$
provided $\lambda>\lambda_0(N)>0$. As a consequence, for $\lambda=\sigma-2$ large enough, the functions 
$$v_{\sigma,+}(t)=e^{-\frac{N-2}{2}t}\cos(\sqrt{\alpha(\sigma-2)}t), \qquad v_{\sigma,-}(t)=e^{-\frac{N-2}{2}t}\sin(\sqrt{\alpha(\sigma-2)}t)
$$
are \textit{formal} point-wise eigenfunctions of the {\it Jacobi operator} of the cone $C_{m,n}$ with eigenvalue $\lambda>\lambda_0(N)$, but they  have infinite energy space, i.e.,
$$\int_1^\infty (v_{\sigma,\pm}(\log s))^2\beta(s)a^{m-1}(s)b^{n-1}(s)ds\ge C\int_1^\infty s^{-(N-2)}s^{-2}s^{N-1}ds=\int_1^\infty s^{-1}ds=\infty.
$$

Set
$$
v_\sigma(t):=v_{\sigma,+}(t)=e^{-\frac{N-2}{2}t}\cos(\sqrt{\alpha(\sigma-2)}t)
$$
and observe that $v_\lambda$ vanishes exactly at the points $t_k:=\frac{\pi}{2\sqrt{\alpha(\sigma-2)}}(1+2k)$.

Now, we set
$${\tt w}_{\sigma,k}(\pe):=w_{\sigma,k}(s):=v_\sigma(\log s)\chi_{[t_k,t_{k+1}]}(\log s),\qquad s=s(\pe),\,\forall\,\pe\in\Sigma.$$
It is possible to see that
\begin{equation}
\begin{aligned}
\mathcal{Q}_\sigma({\tt w}_{\sigma,k})&=\int_\Sigma |\nabla_\Sigma{\tt w}_{\sigma,k}|^2-|A_\Sigma|^2(1+\sigma){\tt w}_{\sigma,k}^2=\\
&=\int_\Sigma (-\Delta_\Sigma{\tt w}_{\sigma,k}-|A_\Sigma|^2{\tt w}_{\sigma,k}){\tt w}_{\sigma,k} -\sigma\int_\Sigma |A_\Sigma|^2{\tt w}_{\sigma,k}^2\\
&= \int_\R \left(-\partial^2_s w_{\sigma,k}-\frac{N-1}{s}\partial_s w_{\sigma,k}-\frac{N-1}{s^2}w_{\sigma,k}\right)w_{\sigma,k}s^{N-1}ds-\sigma\int_\R \frac{N-1}{s^2}w_{\sigma,k}^2 s^{N-1}ds\\
&\hspace{350pt} + o(1)
\\
&=-2\int_\R \frac{N-1}{s^2}w_{\lambda,k}^2 s^{N-1}ds +o(1)
\\
&=-2\int_\Sigma |A_\Sigma|^2{\tt w}_{\sigma,k}^2 +o(1),\end{aligned}
\end{equation}
where the terms $o(1)$ tend to zero as $k\to \infty$, uniformly in $\sigma$. Therefore,  (\ref{est-Qsigma}) is satisfied for any $\phi$ in the infinite dimensional space
$$
X:={\rm span}\{{\tt w}_{\sigma,k}:\,k\ge k_0\}.
$$

Setting
\begin{equation}
\label{cond-sigma-tilde}
\tilde{\sigma}:=W\big(\frac{2\sqrt{2}a_\star}{\eps^2\overline{\beta}}\big)>0,
\end{equation}
where $\overline{\beta}:=\max \limits_{\pe \in \Sigma}\beta(s(\pe))>0$ and where $\beta(s)$ is defined in \eqref{def:Asigma}. Due to the monotonicity of $W$ we have
\begin{equation}\notag
\sqrt{2}w=W\big(\frac{2\sqrt{2}a_\star}{\eps^2\beta}\big)\ge W\big(\frac{2\sqrt{2}a_\star}{\eps^2\overline{\beta}}\big)=\tilde{\sigma}.
\end{equation}

Using  the function $\chi_3$ defined in \eqref{def:chi_j}, we consider the $C^{0,1}(\R^{N+1})$ functions given by
$$
v_{k,\tilde{\sigma}}:=\sum_{j=1}^m (\chi_3{\tt w}_{k,\tilde{\sigma}}(\eps\cdotp)v'_\star)^{\natural}_{j}.$$
The idea is that $v_{k,\tilde{\sigma}}$ is zero far from the graphs of $\h_j(\eps\cdotp)$ over $\Sigma_\eps$, for $1\le j\le k$, and near those graphs it is given by ${\tt w}_{k,\tilde{\sigma}}$ in the direction of the graph, by $v'_\star$ in the orthogonal direction.
\begin{lemma}
\label{lemma-Morse-infinity}
There exists $\bar{k}>0$ and $\bar{\eps}>0$ such that, for any $\eps\in(0,\bar{\eps})$ and $k>\bar{k}$, we have
$$\int_{\R^{N+1}}|\nabla v_{k,\tilde{\sigma}}|^2+(3 u_\eps^2-1) v_{k,\tilde{\sigma}}^2<-c\eps^{2-N}<0.$$
\end{lemma}
\begin{remark}
Lemma \ref{lemma-Morse-infinity} yields that, for $\eps(0,\bar{\eps})$, the Morse index of $u_\eps$ is infinite. In fact the functions ${\tt w}_{k,\tilde{\sigma}}$ can be approximated by functions in $C^{\infty}_c(\R^{N+1})$.
\end{remark}

\begin{proof}
Setting $\phi:={\tt w}_{k,\sigma}$ and using the expansion of the Laplacian in Fermi coordinates, in $X_{\eps,\h_j}^{-1}(\mathcal{N}_{j,\eps})$ we compute
\begin{equation}
\begin{aligned}
\Delta(\phi v'_\star)&=\eps^2\Delta_\Sigma\phi v'_\star+\phi v'''_\star-\eps^2 J_\Sigma\phi\h_j v'_\star-\eps^2|A_\Sigma|^2\phi tv''_\star-2\eps\nabla_\Sigma \h_j\cdotp\nabla_\Sigma\phi v''_\star+\eps^2|\nabla_\Sigma\h_1|^2 \phi v'''_\star\\
&-\eps\Q\phi v''_\star+\eps^2(\ta^{ik}\phi_{ik}+\bi^i\phi_i)v'_\star-\eps^2(\ta^{ik}{\h_1}_{ik}+\bi^i{\h_1}_i)\phi v''_\star
-2\eps \ta^{ik}\phi_i{\h_1}_k v''_\star+\eps^2 \ta^2|\nabla_\Sigma \h_1|^2 \phi v'''_\star
\end{aligned}
\end{equation}
and
\begin{equation} 
\begin{aligned}
(1-3U_1^2)w'_j&=(1-3w_j^2)w'_j-6v_\star v'_j\bigg((2e^{-\sqrt{2}t}-\psi_0(-t))e^{-\sqrt{2}(\h_j-\h_{j-1})}\\
&+(-2e^{-\sqrt{2}t}+\psi_0(t))e^{-\sqrt{2}(\h_{j+1}-\h_j)}
+\eps^2|\nabla_\Sigma\h_j|^2\psi_1(t)+\eps^2|A_\Sigma|^2\psi_2(t)\bigg)\\
&+O(\eps^{2+\gamma}e^{-\varrho|t|}(1+s)^{-(2+\gamma)})\chi_{[t_k,t_{k+1}]}.
\end{aligned}
\end{equation}
Moreover, setting $v_j(t):=w_j(t+\h_{j+1}-\h_j)$ and $u_j(t):=w_j(t+\h_j-\h_{j-1})$, we have
$$(1-3U_1^2)v_j=(1-3v_j^2)v'_j+6\sqrt{2}(1-v_\star^2)e^{\sqrt{2}t}e^{-\sqrt{2}(\h_{j+1}-\h_j)}
+O(\eps^{2+\gamma}e^{-\varrho|t|}(1+s)^{-(2+\gamma)})\chi_{[t_k,t_{k+1}]}$$
for some $\varrho>0$, and
$$(1-3U_1^2)u_j=(1-3u_j^2)u'_j+6\sqrt{2}(1-v_\star^2)e^{-\sqrt{2}t}e^{-\sqrt{2}(\h_{j}-\h_{j-1})}
+O(\eps^{2+\gamma}e^{-\varrho|t|}(1+s)^{-(2+\gamma)})\chi_{[t_k,t_{k+1}]}.$$

Differentiating the equations satisfied by $\psi_0$ and $\psi_2$ and integrating over $\R$ we can see that
\begin{equation}\notag
\begin{aligned}
6\int_\R \psi_2 v_\star v'_\star &=-\frac{1}{2}\int_\R (v'_\star)^2=-\frac{c_\star}{2}\\
6\int_\R (\psi_0-2e^{\sqrt{2}t}) v'_\star &=\sqrt{2}\int_\R 6(1-v_\star^2)e^{\sqrt{2}t}v'_\star=\sqrt{2}a_\star c_\star
\end{aligned}
\end{equation}
so that, multiplying by $v'_j$ and integrating over $\R$, we have
\begin{equation}
\begin{aligned}
\int_\R\left(-\Delta(\phi v'_\star)+(3U_1^2-1)\phi v'_\star\right) v'_\star dt&=-\eps^2c_\star (\Delta_\Sigma\phi+|A_\Sigma|^2\phi)(\eps\cdotp)\\
&-2\sqrt{2} a_0 c_\star e^{-\sqrt{2}(\h_{j-1}(\eps\cdotp)-\h_j(\eps\cdotp))}\phi(\eps\cdotp)+O(\eps^{2+\gamma}(1+s)^{-(2+\gamma)})\chi_{[t_k,t_{k+1}]}.
\end{aligned}
\end{equation}
In conclusion, multiplying by $v:=v_{k,\tilde{\sigma}}$ and integrating over $\R^{N+1}$ we have
\begin{equation}
\begin{aligned}
\int_{\R^{N+1}}\left(-\Delta v+(3u_\eps^2-1)v\right) v d\xi&=-\eps^2 c_\star\left( m\int_{\Sigma_\eps} J_{\Sigma}\phi(\eps\cdotp)\phi(\eps\cdotp)+4\sqrt{2}a_\star (m-1)\int_{\Sigma_\eps} {\tt w}(\eps\cdotp) \phi^2(\eps\cdotp)\right)(1+O(\eps^{\gamma}))\\
\le &\eps^{2-N} c_\star\left( m\int_\Sigma \left(|\nabla_\Sigma \phi|^2-|A_\Sigma|^2\phi^2\right)-4\sqrt{2}a_\star (m-1)\int_{\Sigma} {\tt w} \phi^2\right)(1+O(\eps^{\gamma}))\\ 
\le& C\eps^{2-N}\mathcal{Q}_{\tilde{\sigma}}(\phi)(1+O(\eps^{\gamma}))\\
\le& -C\eps^{2-N}\left(\int_{\Sigma} |A_{\Sigma}|^2\phi^2\right)(1+O(\eps^{\gamma}))<0
\end{aligned}
\end{equation}
where $$c_\star:=\int_\R (v'_\star)^2.
$$

This completes the proof of ${\rm iv.}$ in Theorem \ref{main-th-AC} and thus its proof. 
\end{proof}

\end{document}